\documentclass[a4paper,12pt]{article}
\usepackage{amsmath,amsfonts,amssymb,amsthm,anysize,mathrsfs}%,dsfont}
\usepackage{enumerate}
\usepackage{epsfig}
\usepackage{supertabular}
\usepackage{graphicx}
\usepackage{color}
\usepackage{xcolor}
\usepackage{cases}

\usepackage[nobysame]{amsrefs}

\usepackage{indentfirst}
\usepackage[a4paper, total={160mm, 240mm}]{geometry}

\newtheorem{theorem}{Theorem}[section]
\newtheorem{corollary}[theorem]{Corollary}
\newtheorem{lemma}[theorem]{Lemma}
\newtheorem{proposition}[theorem]{Proposition}
\newtheorem{conjecture}[theorem]{Conjecture}

\newtheorem{thm}{Theorem}[section]
\newenvironment{thmbis}[1]
  {\renewcommand{\thethm}{\ref{#1}$'$}%
   \addtocounter{thm}{-1}%
   \begin{thm}}
  {\end{thm}}

\newtheorem{conj}{Conjecture}[section]
\newenvironment{conjbis}[1]
  {\renewcommand{\theconj}{\ref{#1}$'$}%
   \addtocounter{conj}{-1}%
   \begin{conj}}
  {\end{conj}}

\newtheorem{tthm}{Theorem}[section]
\newenvironment{tthmbis}[1]
  {\renewcommand{\thetthm}{\ref{#1}}%
   \addtocounter{tthm}{-1}%
   \begin{tthm}}
  {\end{tthm}}

\newtheorem{cconj}{Conjecture}[section]

\theoremstyle{definition}
\newtheorem{remark}[theorem]{Remark}
\newtheorem{definition}[theorem]{Definition}
\newtheorem{example}[theorem]{Example}
\newtheorem{problem}[theorem]{Problem}

\newtheorem*{claim*}{Claim}

\usepackage{tikz}

\newcommand\orilongtilde[2]{\vcenter{\hbox{%
\tikz[xscale=#1,yscale=#2]{
\fill plot[smooth] coordinates{
(1,0.061)(0.642,-0.052)(0.417,-0.094)(0.236,-0.094)
(0.097,-0.067)(-0.022,-0.027)(-0.144,0.012)
(-0.283,0.039)(-0.459,0.036)(-0.678,0)(-1,-0.097)
}
-- plot[smooth] coordinates{
(-1,-0.061)(-0.642,0.052)(-0.417,0.094)(-0.236,0.094)
(-0.097,0.067)(0.022,0.027)(0.144,-0.012)
(0.283,-0.039)(0.459,-0.036)(0.678,0)(1,0.097)
}
--cycle;}}}}

\newcommand\longtildeo[1]{%
\pgfmathparse{scalar(#1/2cm)}\let\bilitmp\pgfmathresult
\orilongtilde{\bilitmp}{0.6*\bilitmp^(1/3)}}

\makeatletter
\newcommand\xlongtilde[2][]{%
  \mathrel{\mathop{%
    \setbox\tw@\vbox{\m@th
      \hbox{$\scriptstyle\mkern10mu{#1}$}%
      \hbox{$\scriptstyle\mkern10mu{#2}$}%
    }%
    {\longtildeo{\wd\tw@}}%
    \raise1.16ex\hbox{}%让\ht大点
    }%
  \limits
    \@ifnotempty{#2}{^{#2}}%
    \@ifnotempty{#1}{_{#1}}}%
}
\makeatother

\usepackage{url, hyperref}
\hypersetup{citecolor=blue, linkcolor=blue, colorlinks=true}

\usepackage{url}

\usepackage{setspace}

\usepackage{authblk}

\providecommand{\keywords}[1]{\textbf{\textbf{Keywords.}} #1}

\providecommand{\MSC}[1]{\textbf{\textbf{MSC 2020.}} #1}

\begin{document}

\title{Divisibility of Griesmer Codes}

\author{Haihua Deng\thanks{Research supported by the National Natural Science Foundation of China Grant No. 123B2011.}\quad Hexiang Huang\thanks{Research partially funded by National Key R\&D Program of China under Grant No. 2021YFA1001000, National Natural Science Foundation of China under Grant No. 12231014, and a Taishan scholar program of Shandong Province.}\quad Qing Xiang\thanks{Research partially supported by the National Natural Science Foundation of China Grant No. 12071206, 12131011.}
}

% \date{\today}

\maketitle

% \vspace{-0.8cm}
% \setcounter{tocdepth}{2}

\begin{abstract}
In this paper, we consider Griesmer codes, namely those linear codes meeting the Griesmer bound. Let $C$ be an $[n,k,d]_q$ Griesmer code with $q=p^f$, where $p$ is a prime and $f\ge1$ is an integer. In 1998, Ward proved that for $q=p$, if $p^e|d$, then $p^e|\mathrm{wt}(c)$ for all $c\in C$. In this paper, we show that if $q^e|d$, then $C$ has a basis consisting of $k$ codewords such that the first $\min\left\{e+1,k\right\}$ of them span a Griesmer subcode with constant weight $d$ and any $k-1$ of them span a $[g_q(k-1,d),k-1,d]_q$ Griesmer subcode. Using the $p$-adic algebraic method together with this basis, we prove that if $q^e|d$, then $p^e|\mathrm{wt}(c)$ for all $c\in C$. Based on this fact, using the geometric approach with the aforementioned basis, we show that if $p^e|d$, then $\Delta |{\rm wt}(c)$ for all $c\in C$, where $\Delta=\left\lceil p^{e-(f-1)(q-2)}\right\rceil$.
\end{abstract}

% \subjclass[2020]{11T06, 11T71}
\noindent\MSC{11T71; 11A07, 94B05, 94B65}\\
\noindent\keywords{Griesmer code, Divisible code, Algebraic method, Geometric approach.}

% \vspace{-0.1cm}
% \setcounter{tocdepth}{2}
% \vspace{-0.5cm}
% \setcounter{tocdepth}{2}
% \tableofcontents

\section{Introduction}

The notion of divisible codes was first introduced by Ward \cite{ward1981divisible} in 1981. Divisible codes are linear codes whose codeword weights share a common divisor larger than one. A linear code $C$ is called \textit{$\Delta$-divisible} if the weights $\mathrm{wt}(c)$ of all codewords $c\in C$ are divisible by some positive integer $\Delta$. In such a case, we call $\Delta$ a $\mathit{divisor}$ of $C$. There is an extensive survey on divisible codes by Kurz \cite{kurz2021divisible} (We will use some results from \cite{kurz2021divisible} later). In \cite{ward1981divisible}, Ward proved a fundamental theorem regarding divisible codes, which states that any $\Delta$-divisible code $C$ over $\mathbb{F}_q$ with $\gcd(\Delta,q)=1$ is equivalent to a $\Delta$-fold repetition of some code $C'$, supplemented by sufficiently many 0 coordinates. Therefore, $p$-power divisible codes over $\mathbb{F}_q$ will be of particular interest,  where $p$ is the characteristic of $\mathbb{F}_q$ (we will write $q=p^f$ throughout this paper). Many structural codes such as simplex codes, cyclic codes, and generalized Reed-Muller codes exhibit \textit{good divisibility}, meaning that all weights are divisible by a ``large" $p$-power; see \cites{ax1964zeroes,ward1990weight,ward2001divisible}. Griesmer codes, as a special class of length-optimal codes including the aforementioned simplex codes, also seem to exhibit good divisibility. 

We first recall the Griesmer bound. Let $C$ be an $[n,k,d]_q$ code. In 1960, Griesmer \cite{griesmer1960bound} proved that for binary linear codes (i.e., $q=2$), we have $n\ge\sum_{i=0}^{k-1}\left \lceil \frac{d}{2^i}  \right \rceil$. Later, in 1965, Solomon and Stiffler \cite{solomon1965algebraically} generalized this result, leading to the following theorem.

\begin{theorem}[The Griesmer Bound \cites{griesmer1960bound,solomon1965algebraically}]\label{Griesmer bound}
Let $C$ be an $[n,k,d]_q$ code. Then
\[n\ge\sum_{i=0}^{k-1}\left \lceil \frac{d}{q^i}  \right \rceil.   
\]   
\end{theorem}
\noindent The bound in Theorem \ref{Griesmer bound} will be called \textit{the Griesmer bound} and the quantity on the right-hand side of the above inequality will be denoted by $g_q(k,d)$; that is, $g_q(k,d):=\sum\limits_{i=0}^{k-1}\left \lceil \frac{d}{q^i}  \right \rceil$. Any $[n,k,d]_q$ code $C$ with code length meeting the Griesmer bound will be called a \textit{Griesmer code}.

Besides the simplex codes, there are many other examples of Griesmer codes, such as maximum distance separable (MDS) linear codes, the $1^{\text{st}}$ order Reed-Muller codes and Belov codes (see \cites{belov1974conjecture,dodunekov1986optimal}). We shall present two examples of Griesmer codes arising from finite geometry as follows.

\begin{example}
A \textit{unital} in $\mathrm{PG}(2,q^2)$ is a set $\mathcal{U}$ of $q^3+1$ points such that every line meets $\mathcal{U}$ in 1 or $q+1$ points. Let $G$ be a $3\times (q^3+1)$ matrix with columns $g_1,g_2,\ldots,g_{q^3+1}$ such that $\mathcal{U}:=\{\left\langle g_1 \right\rangle,\left\langle g_2 \right\rangle,\ldots, \left\langle g_{q^3+1} \right\rangle\}$ is a unital in $\mathrm{PG}(2,q^2)$. Then the linear code $C$ with $G$ as a generator matrix is a $[q^3+1,3,q^3-q]_{q^2}$ Griesmer code. Note that $C$ is actually a two-weight code with nonzero weights $q^3-q$ and $q^3$. 
\end{example}

\begin{example}
An \textit{ovoid} in $\mathrm{PG}(3,q)$ is a set $\mathcal{O}$ of $q^2+1$ points such that no three of them are collinear and every hyperplane meets $\mathcal{O}$ in 1 or $q+1$ points. Let $G$ be a $4\times (q^2+1)$ matrix with columns $g_1,g_2,\ldots,g_{q^2+1}$ such that $\mathcal{O}:=\{\left\langle g_1 \right\rangle,\left\langle g_2 \right\rangle,\ldots, \left\langle g_{q^2+1} \right\rangle\}$ is an ovoid in $\mathrm{PG}(3,q)$. Then the linear code $C$ with $G$ as a generator matrix is a $[q^2+1,4,q^2-q]_q$ Griesmer code. Note that $C$ is in fact a two-weight code with nonzero weights $q^2-q$ and $q^2$.    
\end{example}

Let $C$ be a $[g_q(k,d),k,d]_q$ Griesmer code. In 1984, Dodunekov and Manev \cite{dodunekov1984minimum} first showed that in the case when $q=2$, if $2^e|d$, then $2^e|\mathrm{wt}(c)$ for all $c\in C$. In 1998, Ward proved the following divisibility result on Griesmer codes.

\begin{theorem}[Theorem 1 in \cite{ward1998divisibility}]\label{Divisibility over F_p}
Let $C$ be a $[g_p(k,d),k,d]_p$ Griesmer code where $p$ is a prime. If $p^e|d$, then $p^e|\mathrm{wt}(c)$ for all $c\in C$.     
\end{theorem}

The above theorem shows that Griesmer codes over the prime field $\mathbb{F}_p$ have a good divisibility property. However, when the field size is $q=p^f$ $(f>1)$, in general, we cannot hope for a divisibility result as good as in the case when $q=p$; that is, if $C$ is a $[g_q(k,d),k,d]_q$ Griesmer code and $q^e|d$, $q^e$ may not divide $\mathrm{wt}(c)$ for all $c\in C$. For example, the \textit{hexacode} $H$ (see \cite{conway2013sphere} for definition) is a $[6,3,4]_4$ code containing 45 codewords of weight 4, 18 codewords of weight 6 and the zero codeword. The divisor of $H$ is 2 but not 4 although $4|d=4$. Despite this obstruction, Ward proved further divisibility results on Griesmer codes.

\begin{theorem}[\cite{ward2001divisible}]\label{evidences}
Let $C$ be a $[g_q(k,d),k,d]_q$ Griesmer code.
\begin{enumerate}
    \item If $q|d$, then $C$ is $p$-divisible.
    \item In the case when $q=4$, if $2^e|d$, then $2^{e-1}$ is a divisor of $C$.
\end{enumerate}
\end{theorem}

Ward then proposed the following conjecture.

\begin{conjecture}[Conjecture 10 in \cite{ward2001divisible}]\label{Conjecture 1}
Let $C$ be a $[g_q(k,d),k,d]_q$ Griesmer code where $q=p^f$. Suppose that $p^e|d$ and $p^e\ge q$. Then $p^{e+1}/q$ is a divisor of $C$.    
\end{conjecture}

In Section \ref{The Algebraic Method}, we will prove Theorem \ref{First Theorem} using the $p$-adic algebraic method, extending Theorem 1 of \cite{ward1998divisibility} and Proposition 13 of \cite{ward2001divisible}. In Section \ref{Geometric Approach}, we will prove Theorem \ref{Second Theorem} from the geometric point of view.

\begin{theorem}\label{First Theorem}
Let $C$ be a $[g_q(k,d),k,d]_q$ Griesmer code where $q=p^f$. If $q^e|d$, then $p^e | \mathrm{wt}(c)$ for all $c\in C$.
\end{theorem}

\begin{theorem}\label{Second Theorem}
Let $C$ be a $[g_q(k,d),k,d]_q$ Griesmer code where $q=p^f$. If $p^e|d$, then $\left\lceil p^{e-(f-1)(q-2)}\right\rceil$ is a divisor of $C$.     
\end{theorem}

Note that Conjecture \ref{Conjecture 1} and Theorem \ref{First Theorem} can be rephrased as follows. 

\begin{conjbis}{Conjecture 1}
Let $C$ be a $[g_q(k,d),k,d]_q$ Griesmer code where $q=p^f$. If $p^e|d$, then $\left\lceil p^{e-(f-1)} \right\rceil$ is a divisor of $C$.  
\end{conjbis}

\begin{thmbis}{First Theorem}
Let $C$ be a $[g_q(k,d),k,d]_q$ Griesmer code where $q=p^f$. If $p^e|d$, then $p^{\lfloor e/f\rfloor} | \mathrm{wt}(c)$ for all $c\in C$.
\end{thmbis}

The above reformulation of Conjecture \ref{Conjecture 1} makes it clear that Theorem \ref{Second Theorem} can be viewed as a similar, but weaker result compared with Conjecture \ref{Conjecture 1}. These results provide new information about Griesmer codes, which will be helpful in searching for Griesmer codes or proving the nonexistence of Griesmer codes with certain parameters. For nonexistence proofs of certain Griesmer codes, see, e.g. \cites{ward2004sequence,maruta2004nonexistence,kawabata2022nonexistence}.

The rest of this paper is organized as follows: In Section \ref{The Geometric Approach} and Section \ref{Two proofs of the Griesmer bound}, we will give some preliminaries and present two proofs of the Griesmer bound, both of which will be useful later. In Section \ref{A Basis for Griesmer codes section}, we will show that any $[g_q(k,d),k,d]_q$ Griesmer code with $q^e|d$ has a basis consisting of $k$ codewords such that the first $\min\{e+1,k\}$ of them span a Griesmer subcode with constant weight $d$ and any $k-1$ of them span a $[g_q(k-1,d),k-1,d]_q$ Griesmer subcode. In Section \ref{The Algebraic Method}, using the $p$-adic algebraic method together with this basis, we apply a divisibility criterion given by Ward \cite{ward1990weight} to prove Theorem \ref{First Theorem}. Based on Theorem \ref{First Theorem}, again using the basis mentioned above, we prove Theorem \ref{Second Theorem} from the geometric point of view.

\section{Preliminaries}\label{The Geometric Approach}

\subsection{Residual Codes, Projected Codes and Shortened Subcodes}

First, we review some basic definitions.

The \textit{support} of a word $c=(c_1,\ldots,c_n)\in\mathbb{F}_q^n$, denoted by $\mathrm{supp}(c)$, is the set of nonzero coordinates of $c$, i.e., $\mathrm{supp}(c):=\{i \mid c_i\ne0\}$.    

The \textit{residual code} of $C\subseteq\mathbb{F}_q^n$ with respect to a word $a\in \mathbb{F}_q^n$, denoted by $\mathrm{Res}(C,a)$, is the code obtained by puncturing $C$ at $\mathrm{supp}(a)$. The codeword $\mathrm{res}(c,a)\in \mathrm{Res}(C,a)$ denotes the codeword obtained by puncturing $c\in C$ at $\mathrm{supp}(a)$. In other words,
\begin{align*}
\mathrm{Res}(C,a) &= \{ (c_j)_{j \notin \mathrm{supp}(a)} \mid c = (c_1,\ldots,c_n) \in C \},\\
\mathrm{res}(c,a) &= (c_j)_{j \notin \mathrm{supp}(a)}, \quad \forall c \in C.
\end{align*}
\par The \textit{projected code} of $C\subseteq\mathbb{F}_q^n$ with respect to a word $a\in \mathbb{F}_q^n$, denoted by $\mathrm{Proj}(C,a)$, is the code obtained by projecting $C$ at $\mathrm{supp}(a)$. The codeword $\mathrm{proj}(c,a)\in \mathrm{Proj}(C,a)$ denotes the codeword obtained by projecting $c\in C$ at $\mathrm{supp}(a)$. In other words,
\begin{align*}
\mathrm{Proj}(C,a) &= \{ (c_j)_{j \in \mathrm{supp}(a)} \mid c = (c_1,\ldots,c_n) \in C \},\\
\mathrm{proj}(c,a) &= (c_j)_{j \in \mathrm{supp}(a)}, \quad \forall c \in C
\end{align*}

For $1\le i\le n$, the \textit{shortened subcode} $C_i$ is the subcode of $C$ consisting of codewords in $C$ such that the $i$-th coordinate is 0, i.e., 
\begin{align*}
C_i:=\{c=(c_1,\ldots,c_n)\in C \mid c_i=0\}.    
\end{align*}

\begin{remark}
Note that residual codes, projected codes and shortened subcodes of linear codes are still linear codes.    
\end{remark}

\subsection{The Geometric Setting}

The geometric approach towards linear codes embodies geometric intuition and can allow us to prove some interesting results. It is worth noting that Landjev \cite{Landjev2001} used the polynomial method from the geometric point of view to reprove Theorem \ref{Divisibility over F_p}. We will use the geometric approach to prove Theorem \ref{Second Theorem}. First, we recall the geometric setting in the classical projective space $\mathrm{PG}(k-1,q)$ below. 

Let $V=\mathbb{F}_q^k$. The classical projective space $\mathrm{PG}(k-1,q)$ is defined as the geometry in which the point set $\mathcal{P}$ is the set of all 1-dimensional subspaces of $V$, the line set $\mathcal{L}$ is the set of all 2-dimensional subspaces of $V$, ..., the set $\mathcal{H}$ of hyperplanes is the set of all $(k-1)$-dimensional subspaces of $V$, where the incidence relation is the natural inclusion of subspaces.

\begin{definition}
A \textit{multiset} $\mathcal{M}$ (of points) in $\mathrm{PG}(k-1,q)$ is a collection of points in $\mathrm{PG}(k-1,q)$ with multiplicity.     
\end{definition}

We may regard vectors as points if there is no ambiguity. Sometimes it is advantageous to use the definition of a multiset $\mathcal{M}$ of points in $\mathrm{PG}(k-1,q)$ as follows.

\begin{definition}
A \textit{multiset} $\mathcal{M}$ (of points) in $\mathrm{PG}(k-1,q)$ is a mapping $\mathcal{M}:\mathcal{P}\to\mathbb{N}_0:=\mathbb{N}\cup\{0\}$.      
\end{definition}

For each point $P$, the number $\mathcal{M}(P)$ is the multiplicity of the point $P$ in the multiset $\mathcal{M}$. If $\mathcal{M}(P)\in\{0,1\}$ for all points $P\in\mathcal{P}$, then $\mathcal{M}$ is simply a set. We can also add two multisets $\mathcal{M}_1$ and $\mathcal{M}_2$ (regard them as functions) to obtain a multiset $\mathcal{M}_1+\mathcal{M}_2$. For any subset $S\subseteq \mathcal{P}$, we denote $\mathcal{M}(S):=\sum_{P\in S}\mathcal{M}(P)$ and the multiset $\mathcal{M}|_S:\mathcal{P}\to\mathbb{N}_0$ by
\begin{align*}
\mathcal{M}|_S(P):=\begin{cases}
 \mathcal{M}(P), & \text{if } P\in S, \\
 0, & \text{otherwise.}
\end{cases}
\end{align*}

Let $C$ be an $[n,k]_q$ code and $G$ be a generator matrix of $C$. The \textit{effective length} of the linear code $C$, denoted by $n(C)$, is the number of nonzero coordinates of $C$ (or the number of nonzero columns of its generator matrix $G$). If $G$ contains no zero columns, then we will say that $C$ has full length. In such a situation, we may identify $C$ with the multiset $\mathcal{M}$ corresponding to all the columns of a generator matrix $G$ of $C$ since different generator matrices of $C$ will give rise to the ``same" multiset up to equivalence. Formally speaking, two multisets in $\mathrm{PG}(k-1,q)$ are said to be \textit{projectively equivalent} if they are associated with linear codes that are semilinearly isomorphic. We will say that the multiset $\mathcal{M}$ is $\Delta$-divisible when $C$ is $\Delta$-divisible. We will write $\mathcal{M}_C$ to indicate the multiset corresponding to a linear code $C$, or simply $\mathcal{M}$ when $C$ is clear from the context.  Throughout, we always assume that an $[n,k,d]_q$ code $C$ has full length.

Under the assumptions stated above, we can utilize this geometric framework to express the weight $\mathrm{wt}(c)$ of a codeword $c=aG$, where $a$ is a row vector, in a linear code $C$ with a generator matrix $G$ in the following way. 

\begin{proposition}\label{weight}
Let $C$ be an $[n,k]_q$ code with a $k\times n$ generator matrix $G$. Then any codeword $c\in C$ can be uniquely written as $c=aG$ for an $a\in\mathbb{F}_q^k$. Moreover, for $a=0$, $c$ is the zero codeword and so $\mathrm{wt}(c)=0$; for any nonzero $a$, we have
\begin{align*}
\mathrm{wt}(c)=n-\mathcal{M}(H_a),    
\end{align*}
where $\mathcal{M}$ is the multiset corresponding to the columns of $G$ and $H_a:=a^{\perp}$ is the hyperplane consisting of vectors that are orthogonal to $a$.
\end{proposition}
Geometrically, the weight of $c=aG\in C$ is equal to the number of points of $\mathcal{M}$ not in $H_a$. This proposition establishes the connection between the weights of codewords in a linear code $C$, and the intersection sizes of hyperplanes with the multiset corresponding to a generator matrix $G$ of $C$. As a corollary, we obtain the following result.

\begin{corollary}\label{congruence corollary}
Let $C$ be an $[n,k]_q$ code. Then $C$ is $\Delta$-divisible if and only if 
\begin{align*}
\mathcal{M}_C(H)\equiv n \pmod{\Delta}    
\end{align*}
for all hyperplanes $H\in\mathcal{H}$ in $\mathrm{PG}(k-1,q)$.
\end{corollary}

\section{Two Proofs of the Griesmer Bound}\label{Two proofs of the Griesmer bound}

The Griesmer bound was proved by Griesmer in the case when $q=2$, by Solomon and Stiffler for general prime powers $q$. In order to keep the paper self-contained, we will give two proofs for the Griesmer bound in this section.

\subsection{The First Proof of the Griesmer Bound}\label{The first proof}

We first prove a lemma about the residual code $\mathrm{Res}(C,a)$ where $a$ is a minimum weight codeword of $C$.

\begin{lemma}\label{property of the residual code}
Let $C$ be an $[n,k,d]_q$ code with $k\ge2$ and $a$ be a minimum weight codeword of $C$. Then $\mathrm{Res}(C,a)$ is an $[n-d,k-1,d']_q$ code with $d'\ge\left \lceil d/q \right \rceil $.     
\end{lemma}

\begin{proof}

First we claim that $\dim \mathrm{Res}(C,a)=k-1$. Clearly $\mathrm{Res}(C,a)$ is a linear code; so it makes sense to talk about its dimension. Let $P_a$ be the puncturing map
\begin{align*}
P_a: C&\longrightarrow \mathrm{Res}(C,a)\\ 
c&\longmapsto P_a(c):=\mathrm{res}(c,a)
\end{align*} 
It is easy to see that $P_a$ is a linear map. Since $a$ is a minimum weight codeword, we have $\mathrm{ker}(P_a)=\left \langle a \right \rangle $, the 1-dimensional subspace spanned by $a$. Thus, $\dim \mathrm{Res}(C,a)=k-1$.

Next, we prove $d'\ge\left \lceil d/q \right \rceil $. Let $b=(b_1,\ldots,b_n)\in C$ be a codeword which is not a multiple of $a=(a_1,\ldots,a_n)$. Then $\mathrm{res}(b,a)\ne0$. Let $\mathbb{F}_q^{\times}:=\mathbb{F}_q\backslash\{0\}$. For $\alpha\in\mathbb{F}_q^{\times}$, we define
\begin{align*}
A_{\alpha}:=\left | \left\{i: a_i=\alpha b_i\ne0 \right\} \right | ,\text{ where } \alpha\in\mathbb{F}_q^{\times}.
\end{align*}
Then we have
\begin{equation}\label{weight partition}
\mathrm{wt}(\mathrm{res}(b,a))+\sum_{\alpha\in\mathbb{F}_q^{\times}}{A_{\alpha}}=\mathrm{wt}(b).    
\end{equation}
Note that $a-\alpha b\in C$, we have 
\begin{equation*}
\mathrm{wt}(a-\alpha b)=d-A_{\alpha}+\mathrm{wt}(\mathrm{res}(b,a)) \ge d,  
\end{equation*}
from which it follows that
\begin{equation}\label{upper bound of agreement}
A_{\alpha}\le \mathrm{wt}(\mathrm{res}(b,a)).    
\end{equation}
Combining (\ref{weight partition}) with (\ref{upper bound of agreement}), we obtain
\begin{equation}\label{weight of b and weight of res(b,a)}
\mathrm{wt}(b)\le q\cdot \mathrm{wt}(\mathrm{res}(b,a)).    
\end{equation}
Since $\mathrm{wt}(b)\ge d$, we have $q\cdot \mathrm{wt}(\mathrm{res}(b,a))\ge d$ and so $\mathrm{wt}(\mathrm{res}(b,a))\ge \left \lceil d/q \right \rceil $.    
\end{proof}

\begin{proof}[The First Proof of Theorem \ref{Griesmer bound}]
We proceed by induction on the dimension $k$. Clearly the bound holds for $k=1$. Assume that the conclusion holds for linear codes of dimension less than $k$. Applying the induction hypothesis and Lemma \ref{property of the residual code} on $\mathrm{Res}(C,a)$, we have
\begin{align*}
n-d\ge \sum_{i=0}^{(k-1)-1}\left \lceil \frac{d'}{q^{i}}\right \rceil\ge\sum_{i=0}^{(k-1)-1}\left \lceil \frac{\left \lceil \frac{d}{q} \right \rceil }{q^i} \right \rceil=\sum_{i=1}^{k-1}\left \lceil \frac{d}{q^{i}}\right \rceil.     
\end{align*}
The proof is now complete.
\end{proof}

\subsection{The Second Proof of the Griesmer Bound}\label{The first proof}

By $\gamma(\mathcal{M})$ we denote the maximum point multiplicity of a multiset of points $\mathcal{M}$ in $\mathrm{PG}(k-1,q)$. A linear code $C$ is called \textit{projective} if $\gamma(\mathcal{M}_C)=1$.

\begin{lemma}\label{lower bound of maximum multiplicity}
Let $C$ be an $[n,k,d]_q$ code and $\mathcal{M}$ be the corresponding multiset of $C$. Then
\begin{align*}
\gamma(\mathcal{M})\ge\left \lceil \frac{d}{q^{k-1}}  \right \rceil.    
\end{align*}      
\end{lemma}

\begin{proof}
Suppose that $\mathcal{M}$ corresponds to a generator matrix $G=(g_1,\ldots,g_n)$ of $C$ where $g_i=(g_{1,i},\ldots,g_{k,i})^{\top}\in\mathbb{F}_q^k\backslash\{0\}$, $i=1,\ldots,n$. Without loss of generality, assume that $\left \langle g_1 \right \rangle \in\mathrm{PG}(k-1,q)$ has the maximum multiplicity. Suppose $g_{\ell,1}\ne0$ for some $\ell$ since one of $g_{i,1}$ $(i=1,\ldots,k)$ must be nonzero. Note that the $\ell$-th row of $G$ has weight $\mathrm{wt}(g_{\ell,1},\ldots,g_{\ell,n})\ge d$. Therefore, among the columns indexed by $\mathrm{supp}(g_{\ell,1},\ldots,g_{\ell,n})$, there are at least $\left \lceil\mathrm{wt}(g_{\ell,1},\ldots,g_{\ell,n})/q^{k-1}\right \rceil \ge\left \lceil d/q^{k-1} \right \rceil $ columns of $G$ which are the same in $\mathrm{PG}(k-1,q)$ by the pigeonhole principle because there are $q^{k-1}$ different points whose $\ell$-th coordinate is nonzero. The proof is now complete.
\end{proof}

\begin{proof}[The Second Proof of Theorem \ref{Griesmer bound}]
We again use induction on the dimension $k$. Clearly, the bound holds for $k=1$. Assume that the conclusion holds for linear codes of dimension less than $k$. Suppose that $g_{i_1},\ldots,g_{i_{\gamma(\mathcal{M})}}$ are the same up to a scalar factor, i.e., $\left \langle g_{i_1} \right \rangle =\cdots=\langle g_{i_{\gamma(\mathcal{M})}} \rangle =:\left \langle g \right \rangle $. Consider the shortened subcode $C_{\left \langle g \right \rangle }$ of $C$ as follows:
\begin{align*}
C_{\left \langle g \right \rangle }:=\left\{c=(c_1,\ldots,c_n)\in C:c_i=0,\forall i\in\{i_1,\ldots,i_{\gamma(\mathcal{M})}\}\right\}.
\end{align*}
Since $C_{\left \langle g \right \rangle }\subsetneqq C$, $\dim C_{\left \langle g \right \rangle }\le k-1$. Moreover, since $\left \langle g_{i_1} \right \rangle =\cdots=\langle g_{i_{\gamma(\mathcal{M})}} \rangle $, the dimension decreases by 1 and so $\dim C_{\left \langle g \right \rangle }=k-1$. Note that $C_{\left \langle g \right \rangle }$ is an $\left[n,k-1,d'\right]_q$ code with $n(C_{\left \langle g \right \rangle })=n-\gamma(\mathcal{M})$ and $d'\ge d$. By induction hypothesis, we have
\begin{align*}
n-\gamma(\mathcal{M})\ge \sum_{i=0}^{(k-1)-1} \left \lceil d'/q^i \right \rceil \ge \sum_{i=0}^{(k-1)-1} \left \lceil d/q^i \right \rceil   
\end{align*}
and by Lemma \ref{lower bound of maximum multiplicity}, $\gamma(\mathcal{M})\ge\left \lceil \frac{d}{q^{k-1}}  \right \rceil$, then we see that
\begin{align*}
n\ge \sum_{i=0}^{k-2} \left \lceil d/q^i \right \rceil+\gamma(\mathcal{M})\ge \sum_{i=0}^{k-2} \left \lceil d/q^i \right \rceil+\left \lceil d/q^{k-1} \right \rceil=\sum_{i=0}^{k-1} \left \lceil d/q^i \right \rceil,
\end{align*}
which proves the case when the dimension is $k$. Now the proof is complete.   
\end{proof}

From these two proofs of the Griesmer bound, we derive the following result.

\begin{theorem}\label{residual code and shortened code}
Let $C$ be a $[g_q(k,d),k,d]_q$ Griesmer code where $k\ge 2$. Then
\begin{enumerate}
    \item For any codeword $a$ with minimum weight $\mathrm{wt}(a)=d$, the residual code of $C$ with respect to $a$, $\mathrm{Res}(C,a)$ is a $\left[g_q\left(k-1,\left \lceil \frac{d}{q}  \right \rceil \right),k-1,\left \lceil \frac{d}{q}  \right \rceil \right]_q$ Griesmer code.
    \item The shortened subcode $C_{\left \langle g \right \rangle }$ is a $\left[g_q(k-1,d),k-1,d\right]_q$ Griesmer code, where $\left \langle g \right \rangle\in\mathcal{M}_C$ is any point attaining the maximum point multiplicity, i.e., $\mathcal{M}_C(\left \langle g \right \rangle)=\gamma(\mathcal{M}_C)$.
\end{enumerate}  
\end{theorem}

As a corollary, we obtain the maximum point multiplicity of Griesmer codes.

\begin{corollary}\label{the maximum point multiplicity of Griesmer codes}
Let $C$ be an $[n,k,d]_q$ Griesmer code. Then $\gamma(\mathcal{M}_C)=\left \lceil d/q^{k-1} \right \rceil $.    
\end{corollary}

\begin{corollary}
Let $C$ be an $[n,k,d]_q$ Griesmer code. Then $C$ is projective if and only if $d\le q^{k-1}$.     
\end{corollary}

\section{A Basis for Griesmer Codes}\label{A Basis for Griesmer codes section}

Before proceeding to the proofs of Theorem \ref{First Theorem} and Theorem \ref{Second Theorem}, we shall show that any $[g_q(k,d),k,d]_q$ Griesmer code with $q^e|d$ has a basis (which we implicitly consider as ordered) consisting of $k$ codewords such that the first $\min\{e+1,k\}$ of them span a Griesmer subcode with constant weight $d$ and any $(k-1)$ of them span a $[g_q(k-1,d),k-1,d]_q$ Griesmer subcode. We will gather here some related results first.

\begin{proposition}[Corollary 5 in \cite{ward1998divisibility}]\label{constant weight}
Suppose that $C$ is a $[g_q(k,d),k,d]_q$ Griesmer code and that $q^{k-1}|d$. Then all nonzero codeword of $C$ have weight $d$.  
\end{proposition}

We remark here that constant weight linear codes are projectively equivalent to a replication of Hamming duals, as shown in \cite{bonisoli1984every}, with a new proof provided in \cite{ward1996characters}.

\begin{proposition}\label{preimage of minimum weight codeword in Res(C,a)}
Let $C$ be a $[g_q(k,d),k,d]_q$ Griesmer code and $\mathrm{Res}(C,a)$ be the residual code of $C$ with respect to a minimum weight codeword $a\in C$. Let $P_a$ be the puncturing map as in the proof of Lemma \ref{property of the residual code} and $c'$ be a minimum weight codeword in $\mathrm{Res}(C,a)$. 
\begin{enumerate}
    \item If $q|d$, then any preimage of $c'$ under $P_a$ must be a minimum weight codeword in $C$. In other words, the number of the preimages of $c'$ with minimum weight $d$ is $q$. 
    \item If $q\nmid d$, then there exists a preimage of $c'$ having minimum weight $d$ in $C$, moreover, the number of these preimages is at least $d+q-\left \lceil d/q \right \rceil q$. 
\end{enumerate}   
In summary, there always exists a preimage of $c'$ having minimum weight $d$ in $C$.
\end{proposition}

\begin{proof} 
There are a total of $q$ preimages of $c'\in \mathrm{Res}(C,a)$ as seen in the proof of Lemma \ref{property of the residual code}. Now we take an arbitrary preimage $b\in C$ of $c'$, i.e., $\mathrm{res}(b,a)=c'$.
\begin{enumerate}
    \item When $q|d$: According to (\ref{weight of b and weight of res(b,a)}),
\begin{align*}
\mathrm{wt}(b)\le q\cdot\mathrm{wt}(\mathrm{res}(b,a))=q\cdot\mathrm{wt}(c')=q\cdot\left \lceil \frac{d}{q}  \right \rceil =q\cdot\frac{d}{q}=d.   
\end{align*}
Thus, $\mathrm{wt}(b)=d$. This proves the first part.
    \item When $q\nmid d$: Write $d=d_1q+d_0$ where $1\le d_0\le q-1$. Then $\mathrm{wt}(c')=\left \lceil \frac{d}{q}  \right \rceil =d_1+1$. As in the proof of Lemma \ref{property of the residual code}, $A_{\alpha}\le d_1+1$ and if there exists an $\alpha\in\mathbb{F}_q$ such that $A_{\alpha}=d_1+1$, then $\mathrm{wt}(b-\alpha a)=d$. Now suppose that $A_{\alpha}\le d_1$ for all $\alpha\in\mathbb{F}_q$. According to (\ref{weight partition}),
    \begin{align*}
    \mathrm{wt}(b)=\mathrm{wt}(\mathrm{res}(b,a))+\sum_{\alpha\in\mathbb{F}_q^{\times}}{A_{\alpha}}\le d_1+1+(q-1)d_1=d_1q+1.
    \end{align*}
    On the other hand, $\mathrm{wt}(b)\ge d=d_1q+d_0\ge d_1q+1$, so $\mathrm{wt}(b)=d$ and $d=d_1q+1$ in this case. Now we know that there exists a preimage $b\in C$ of $c'$ such that $\mathrm{wt}(b)=d$. We consider (\ref{weight partition}) again,
    \begin{align*}
    \mathrm{wt}(b)=\mathrm{wt}(\mathrm{res}(b,a))+\sum_{\alpha\in\mathbb{F}_q^{\times}}{A_{\alpha}}=d_1+1+\sum_{\alpha\in\mathbb{F}_q^{\times}}{A_{\alpha}}=d_1q+d_0    
    \end{align*}
    by $A_{\alpha}\le d_1+1$ we see that there are at least $d_0-1$ elements $\alpha\in\mathbb{F}_q^{\times}$ such that $A_{\alpha}=d_1+1$. Hence, there are at least $d_0$ preimages of $c'$ with minimum weight $d$. 
\end{enumerate}
The proof is now complete.
\end{proof}

\begin{remark}
The first part of Proposition \ref{preimage of minimum weight codeword in Res(C,a)} can also be proved by using Proposition \ref{constant weight} with $C=\left\langle a,b \right\rangle$ and $k=2$.
\end{remark}

\begin{proposition}[c.f. Proposition 4 in \cite{ward1998divisibility}]\label{supplementary}
Let $C$ be a $[g_q(k,d),k,d]_q$ Griesmer code. Then for any codeword $a\in C$ with minimum weight $\mathrm{wt}(a)=d$, there is a supplementary $[g_q(k-1,d),k-1,d]_q$ Griesmer subcode $C'$, i.e., $\left \langle a \right \rangle\cap C'=0$ and $\left \langle a \right \rangle+C'=C$.    
\end{proposition}

\begin{lemma}\label{bijection between C' and Res(C,a)}
Let $C$ be a $[g_q(k,d),k,d]_q$ Griesmer code and $C'$ be a $[g_q(k-1,d),k-1,d]_q$ Griesmer subcode of $C$ which is supplementary to a minimum weight codeword $a$. Then the puncturing map $P_a:C\to \mathrm{Res}(C,a)$ restricted to $C'$ is a bijection.     
\end{lemma}

\begin{proof}
Since $\dim C'=\dim \mathrm{Res}(C,a)=k-1$, it suffices to show that $\ker(P_a|_{C'})=\{0\}$. This can be seen by noticing that $\ker(P_a|_{C'})=C'\cap \left \langle a \right \rangle =\{0\}$.   
\end{proof}

\begin{theorem}\label{A basis for Griesmer codes}
Let $C$ be a $[g_q(k,d),k,d]_q$ Griesmer code and $q^e|d$. Then there exists a generator matrix $G$ of $C$ such that the first $\min\{e+1,k\}$ rows of $G$ span a Griesmer subcode with constant weight $d$. Moreover, any $k-1$ rows of $G$ generate a $[g_q(k-1,d),k-1,d]_q$ Griesmer subcode.      
\end{theorem}

\begin{proof}
By Proposition \ref{supplementary}, we can choose a generator matrix $G$ of $C$ such that the first row $a_1$ has minimum weight $d$, moreover, the last $k-1$ rows of $G$ span a $[g_q(k-1,d),k-1,d]_q$ Griesmer subcode $C'$ that is supplementary to the first row $a_1$ of $G$. Note that at this moment we just fix the first row $a_1$ for $G$ so the last $k-1$ rows still need to be determined and will be determined later. Since the effective lengths of $C$ and $C'$ are respectively $n=g_q(k,d)$ and $n'=g_q(k-1,d)$, then there are exactly
\begin{align*}
g_q(k,d)-g_q(k-1,d)=\sum\limits_{i=0}^{k-1}\left \lceil \frac{d}{q^i}  \right \rceil-\sum\limits_{i=0}^{k-2}\left \lceil \frac{d}{q^i}  \right \rceil=\left \lceil \frac{d}{q^{k-1}}  \right \rceil     
\end{align*}   
columns of $G$ having the form $(\alpha,0,\ldots,0)^{\top}\in\mathbb{F}_q^k$ where $\alpha\in\mathbb{F}_q^{\times}$.

Now we consider the residual code of $C$ with respect to the first row $a_1$, namely $\mathrm{Res}(C,a_1)$. By Theorem \ref{residual code and shortened code}, we know that $\mathrm{Res}(C,a_1)$ is a $\left[g_q\left(k-1,\left \lceil \frac{d}{q}  \right \rceil \right),k-1,\left \lceil \frac{d}{q}  \right \rceil \right]_q$ Griesmer code. For this residual code $\mathrm{Res}(C,a_1)$, by Proposition \ref{preimage of minimum weight codeword in Res(C,a)}, there exists a codeword $b\in C$ such that $\mathrm{wt}(\mathrm{res}(b,a_1))=\left \lceil \frac{d}{q}  \right \rceil $ and then we apply the above operation again so that there are exactly
\begin{align*}
g_q\left(k-1,\left \lceil \frac{d}{q}  \right \rceil\right)-g_q\left(k-2,\left \lceil \frac{d}{q}  \right \rceil\right)=\sum\limits_{i=0}^{k-1-1}\left \lceil \frac{\left \lceil \frac{d}{q}  \right \rceil}{q^i}  \right \rceil-\sum\limits_{i=0}^{k-2-1}\left \lceil \frac{\left \lceil \frac{d}{q}  \right \rceil}{q^i}  \right \rceil=\left \lceil \frac{d}{q^{k-1}}  \right \rceil      
\end{align*}
columns of $G$ having the form $(0,\beta,\ldots,0)^{\top}\in\mathbb{F}_q^k$ where $\beta\in\mathbb{F}_q^{\times}$. Moreover, Lemma \ref{bijection between C' and Res(C,a)} guarantees the existence of some $a_2\in C'$ that satisfies the same key properties as $b$ does in the above construction. Iterating this process, we can produce a basis $\{a_1,a_2,\ldots,a_k\}$ of $C$ to obtain a generator matrix $G$ with exactly $\left \lceil \frac{d}{q^{k-1}}  \right \rceil$ columns which are multiple of $e_j=(0,\ldots,0,1,0,\ldots,0)^{\top}$ with $1$ in the $j$-th place and 0 in other places for all $j=1,\ldots,k$. Furthermore, if $e\ge k-1$, by Proposition \ref{constant weight}, $C$ is a constant weight code; if $e<k$, by construction, we see that the effective length of the subcode $\mathrm{span}_{\mathbb{F}_q}\{a_1,a_2,\ldots,a_{e+1}\}$ is $g_q(e+1,d)$. Therefore, $\mathrm{span}_{\mathbb{F}_q}\{a_1,a_2,\ldots,a_{e+1}\}$ is a $[g_q(e+1,d),e+1,d]_q$ Griesmer code and so by Proposition \ref{constant weight}, it is a constant weight subcode of $C$.   

Finally, we show that any $k-1$ rows of $G$ span a Griesmer subcode. Let $S$ be a set of some $k-1$ rows of $G$ and $D$ be the subcode spanned by $S$. If $S$ contains the first row of $G$, then the minimum weight of $D$ is $d$.  Moreover, $D$ has effective length $n-\left \lceil \frac{d}{q^{k-1}}  \right \rceil=g_q(k-1,d)$ by the above construction, and hence this code $D$ must be a $[g_q(k-1,d),k-1,d]_q$ Griesmer subcode of $C$. If $D$ is spanned by the last $k-1$ rows of $G$, then $D$ is exactly the supplementary Griesmer subcode $C'$. We are done.
\end{proof}

For a $[g_q(k,d),k,d]_q$ Griesmer code $C$, we may call the corresponding multiset $\mathcal{M}_C$ a \textit{Griesmer multiset} and $\gamma(\mathcal{M}_C)=\left \lceil d/q^{k-1} \right \rceil $. In \cite{ward2001divisible}, a point $g$ of a Griesmer multiset with the maximum point multiplicity is called an \textit{endpoint}. From the proof of Theorem \ref{A basis for Griesmer codes}, we see that there are at least $k$ endpoints of an $[n,k,d]_q$ Griesmer code. The following proposition gives a better lower bound for the number of endpoints of a Griesmer code.

\begin{proposition}[Lemma 11 in \cite{ward2001divisible}]
Let $C$ be an $[n,k,d]_q$ Griesmer code. Then the number of endpoints of $C$ is at least $g_q(k,t)$, where $t=d-(\left \lceil d/q^{k-1} \right \rceil-1)q^{k-1}$.    
\end{proposition}

\section{The Algebraic Proof of \text{Theorem \ref{First Theorem}}}\label{The Algebraic Method}

\subsection{Ward's Divisibility Criterion}

In 1990, Ward developed the \textit{weight polarization} method \cite{ward1990weight}, and, as an application, gave a new proof of Ax's theorem \cite{ax1964zeroes}. In 1998, Ward empolyed this technique to prove Theorem \ref{Divisibility over F_p} ($q=p$). However, it seems difficult to apply this method when $q=p^f$ with $f\ge 2$. One of the difficulties lies in analyzing the $p$-adic valuation of a sum involving certain binomial coefficients. To overcome this difficulty, we compute some specific $p$-adic valuations. More precisely, the sum mentioned above has the following form:  
\begin{align*}
c(r,s;1):=\binom{rq}{s}+\binom{rq}{(q-1)+s}+\cdots+\binom{rq}{a(q-1)+s}+\cdots=\sum_{i\equiv s\pmod{q-1}}\binom{rq}{i},
\end{align*}    
where $1\le r,s\le q-1$ and the above summation is in fact a finite sum because the binomial coefficient $\binom{rq}{a(q-1)+s}$ will be 0 if $a\ge q$. The summation of a similar form has been extensively studied, see \cites{sun1995combinatorial,sun2002sum,sun2007congruences}.

Below, we will give a brief overview of weight polarization in the context of $p$-adic numbers.

Let $\mathbb{Z}_p$ be the ring of $p$-adic integers and $\mathbb{Q}_p$ be the field of $p$-adic numbers. For $q=p^f$, $\mathbb{Q}_p$ has a unique unramified extension $\mathbb{Q}_p(\xi_{q-1})$ whose residue class field is $\mathbb{F}_q$, where $\xi_{q-1}$ is a primitive $(q-1)$-st root of unity. Let $\mathbb{Z}_p[\xi_{q-1}]$ be the ring of integers of $\mathbb{Q}_p(\xi_{q-1})$. Note that $\mathbb{Z}_p[\xi_{q-1}]$ is a local ring and the maximal ideal of $\mathbb{Z}_p[\xi_{q-1}]$ is $(p)$. $\mathbb{Q}_p(\xi_{q-1})$ is the splitting field of $X^q-X$ over $\mathbb{Q}_p$, and $\mathbb{Z}_p[\xi_{q-1}]$ contains the full group $U_{q-1}$ of $(q-1)$-st roots of unity. Let $U=U_{q-1}\cup\{0\}$.

We have a canonical map
\[\varphi:\mathbb{Z}_p[\xi_{q-1}]\to \mathbb{Z}_p[\xi_{q-1}]/(p) \cong \mathbb{F}_q.\]
Note that $\varphi$ is a bijection when restricted to $U$. For each $x$ in $\mathbb{Z}_p[\xi_{q-1}]$, there is a unique $T(x)\in U$ satisfying
\begin{align*}
x\equiv T(x)\pmod{p}.   
\end{align*}
The element $T(x)$ is called the \textit{Teichm\"uller representative} of $x$, and the set $U$ consists of Teichm\"uller representatives. When there is no ambiguity, if $x\in \mathbb{F}_q$, we also use $T(x)$ to denote the \textit{Teichm\"uller lift} of $x$, namely the element in $U$ such that $\varphi(T(x)) = x$. For a vector $c=(c_1,\ldots,c_n)\in\mathbb{F}_q^n$, let $T(c)$ be the Teichm\"uller lift of $c$:
\begin{align*}
T(c)=(T(c_1),\ldots, T(c_n)).    
\end{align*}
If $R$ is any commutative ring, and $a_i\in R^n$ for $i=1,\ldots,r$, the component-wise product (sometimes called Schur product or Hadamard product) of $a_1,\ldots,a_r$ will be denoted by $a_1\circ\cdots\circ a_r$. For $a\in R^n$, we will write $\overbrace{a\circ\cdots\circ a}^{r}$ as $a^{\circ(r)}$, and use $\sigma(a)$ to denote the sum of the components of $a$. By $S_p(x)$ we denote the digit sum of $p$-adic expansion of $x\in\mathbb{N}$. By $\nu_p(x)$ we denote the $p$-adic valuation of $x\in \mathbb{Z}_p[\xi_{q-1}]$. The divisibility criterion we will employ is Theorem 5.3 of \cite{ward1990weight}:

\begin{theorem}[Theorem 5.3 in \cite{ward1990weight}]\label{divisibility criterion}
Let $C$ be a linear code in $\mathbb{F}_q^n$ with additive spanning set $B$. Then $p^e$ is a divisor of $C$ if and only if 
\begin{align*}
e\le\frac{1}{p-1}\sum_{i=1}^m S_p(r_i)-f+\nu_p\left(\sigma\left(T(b_1)^{\circ(r_1)}\circ\cdots \circ T(b_m)^{\circ(r_m)}\right)\right)
\end{align*}
for all $M=(r_1,\ldots,r_m)$ with $\sum_{i=1}^m r_i\equiv0\pmod{q-1}$ and all choices of $b_1,\ldots,b_m$ in $B$.
\end{theorem}

Let $C$ be a $[g_q(k,d),k,d]_q$ Griesmer code and $\{a_1,\ldots,a_k\}$ be a basis as constructed in Theorem \ref{A basis for Griesmer codes}, i.e., $\mathrm{span}_{\mathbb{F}_q}\{a_1,a_2,\ldots,a_{e+1}\}$ is a constant weight subcode and any $k-1$ of them span a $[g_q(k-1,d),k-1,d]_q$ Griesmer subcode. Note that $\mathrm{span}_{\mathbb{F}_q}\{a_1,a_2\}$ is a $[g_q(2,d),2,d]_q$ subcode of $C$. If $q|d$, then by Proposition \ref{constant weight}, $\mathrm{span}_{\mathbb{F}_q}\{a_1,a_2\}$ is a constant weight code. Let 
\begin{align*}
B=\{\lambda a\mid\lambda\in\mathbb{F}_q^{\times}\}\cup \{\lambda a_2\mid\lambda\in\mathbb{F}_q^{\times}\}\cup\cdots\cup\{\lambda a_k\mid\lambda\in\mathbb{F}_q^{\times}\}    
\end{align*}
where $a=a_1+\alpha a_2\in C\backslash\mathrm{span}_{\mathbb{F}_q}\{a_2,\ldots,a_k\}$ for some $\alpha\in\mathbb{F}_q$. Then $B$ is an additive spanning set of $C$. Moreover, if $q|d$, we have $\mathrm{wt}(a)=d$ for all $\alpha\in\mathbb{F}_q$. Let $A_i=T(a_i)$, and $\alpha'=T(\alpha)$ for short. If we consider $B$ as the additive spanning set of $C$, then $p^e$ is a divisor of $C$ if and only if 
\begin{equation}\label{divisibility condition}
p^{e-\frac{1}{p-1}\sum_{i=1}^k S_p(r_i)+f}\;\big|\;\sigma\left(T(A_1+\alpha' A_2)^{\circ(r_1)}\circ A_2^{\circ(r_2)}\circ \cdots\circ A_k^{\circ(r_k)}\right)   
\end{equation}
for all $M=(r_1,\ldots,r_m)$ with $\sum_{i=1}^m r_i\equiv0\pmod{q-1}$, where all scalar multiples $\lambda$'s are omitted since they do not change the $p$-adic valuation. 

First of all, we demonstrate how to expand the Schur product in (\ref{divisibility condition}), which can be found in \cite{ward1990weight}. Let $1\le r\le q-1$ and $X,Y\in U=U_{q-1}\cup\{0\}$. After expanding $(X+Y)^{rq^n}$ with $X^{q}=X$ and $Y^{q}=Y$, we obtain the following equality:
\begin{equation*}
(X+Y)^{rq^n}-X^r-Y^r=\sum_{k=1}^{q-1}c(r,k;n)X^{q-1+r-k}Y^k,    
\end{equation*}
where $c(r,k;n)$ is a sum of some binomial coefficients. Let $n\to\infty$, we will see that 
\begin{align*}
T(X+Y)^r=X^r+Y^r+\sum_{k=1}^{q-1}c(r,k)X^{q-1+r-k}Y^k,    
\end{align*}
where $c(r,k)=\lim_{n \to \infty} c(r,k;n)$, with the limit taken in the field of $p$-adic numbers. A similar result for Schur products also holds as follows.

\begin{proposition}[c.f. Proposition 8 in \cite{ward1998divisibility}]\label{Expansion}
Let $x,y\in\mathbb{F}_q^n$, and let $T(x)=X$, $T(y)=Y$. Then for $1\le r\le q-1$,
\begin{align*}
T(x+y)^{\circ(r)}=X^{\circ(r)}+Y^{\circ(r)}+\sum_{i=1}^{r-1}c(r,i)X^{\circ(r-i)}\circ Y^{\circ(i)}+\sum_{i=r}^{q-1}c(r,i)X^{\circ(q-1+r-i)}\circ Y^{\circ(i)}.    
\end{align*}
\end{proposition}

\subsection{The Algebraic Method}

In this subsection, we will prove Theorem \ref{First Theorem} by applying Theorem \ref{divisibility criterion} with the basis $\{a_1,\ldots,a_k\}$ and the corresponding additive spanning set $B$ as mentioned above. Before that, we need some results below on the $p$-adic valuation of binomial coefficients.

\begin{theorem}[Kummer's theorem, \cite{kummer1852erganzungssatze}]\label{Kummer's theorem}
Given integers $n\ge m\ge 0$ and a prime number $p$, the $p$-adic valuation of the binomial coefficient $\binom{n}{m}$ is given by
\begin{align*}
\nu_p\binom{n}{m}=\frac{S_p(m)+S_p(n-m)-S_p(n)}{p-1},    
\end{align*}
which is equal to the number of carries when $m$ is added to $n-m$ in base $p$.
\end{theorem}

\begin{corollary}\label{shifted p-adic valuation}
The $p$-adic valuation of $\binom{np^s}{mp^s}$ is equal to the $p$-adic valuation of $\binom{n}{m}$.    
\end{corollary}

\begin{proof}
By Theorem \ref{Kummer's theorem}, we have
\begin{align*}
\nu_p\binom{np^s}{mp^s}=\frac{S_p(mp^s)+S_p(np^s-mp^s)-S_p(np^s)}{p-1}=\frac{S_p(m)+S_p(n-m)-S_p(n)}{p-1}=\nu_p\binom{n}{m}
\end{align*}
since $S_p(\ell\cdot p^t)=S_p(\ell)$ for any positive integers $\ell,t$.
\end{proof}

\begin{lemma}\label{the number of carries in certain z=x+y version 1}
     Let $x,y,z$ be nonnegative integers such that $z=x+y$. If $\nu_p(x) = m$ and $\nu_p(z)=n$ for some $n\ge m$, then $\nu_p\binom{z}{x}\ge n-m$; in other words, there are at least $n-m$ carries in the base-$p$ addition of $z = x+y$.
\end{lemma}

\begin{proof}
    Consider the $p$-adic valuation of the binomial coefficient $\binom{z}{x}$, we have
    \[\nu_p\binom{z}{x} = \nu_p\left(\frac{z}{x}\binom{z-1}{x-1}\right)\ge \nu_p\left(\frac{z}{x}\right)=\nu_p(z)-\nu_p(x)=n-m.\]
    By Theorem \ref{Kummer's theorem}, the number of carries in the base-$p$ addition of $z = x+y$ is at least $n-m$.
\end{proof}

\begin{proposition}\label{the number of carries in certain z=x+y version 2}
    Let $x,y,z$ be nonnegative integers such that $z=x+y$. Write $x = x'+x'',z = z'+z''$, where $\nu_p(x'),\nu_p(z')\geq n$ and $0\le x'',z''\le p^n-1$. Assume that the $p$-adic expansions of $x''$ and $z''$ are $x''=\sum_{i=0}^{n-1}x_ip^i,z'' = \sum_{i = 0}^{n-1}z_ip^i$. If $x_i\ge z_i$ for $i = 1,\ldots,n-1$ and $x_0>z_0$, then there are at least $n$ carries in the base-$p$ addition of $z = x+y$.
   
\end{proposition}

\begin{proof}
    We claim that the number of carries in the base-$p$ addition of $z = x+y$ is equal to the number of carries in the base-$p$ addition of $z' = w+y$, where $w = x-z'' = x'+(x''-z'')$. 
    Note that if there are no carries in the base-$p$ addition of $c = a+b$, then $S_p(c) = S_p(a)+S_p(b)$ by Theorem \ref{Kummer's theorem}. By repeatedly applying this observation, we see that
    \begin{align*}
        \nu_p\binom{z}{x} &= \frac{S_p(x)+S_p(y)-S_p(z)}{p-1} \\
        &= \frac{S_p(x')+S_p(x'')+S_p(y)-S_p(z')-S_p(z'')}{p-1} \\
        &=\frac{S_p(x')+S_p(x''-z'')+S_p(y)-S_p(z')}{p-1}\\
        &=\frac{S_p(x'+x''-z'')+S_p(y)-S_p(z')}{p-1}\\
        &=\frac{S_p(w)+S_p(y)-S_p(z')}{p-1}\\
        &=\nu_p\binom{z'}{w}.
    \end{align*}
    Hence the claim holds. Note that $w=x'+(x''-z'')$ and the lowest digit of $w$ is $x_0-z_0>0$, thus $\nu_p(w)=0$. Since $\nu_p(z')\ge n$, by Lemma \ref{the number of carries in certain z=x+y version 1}, we know that the number of carries in the base-$p$ addition of $z'=w+y$ is
    \begin{align*}
    \nu_p\binom{z'}{w}\ge\nu_p(z')\ge n.    
    \end{align*} 
    The proof is now complete.
\end{proof}

\begin{lemma}\label{valuation of binom}
For $0\le k\le q-1$, the $p$-adic valuation of $\binom{q-1}{k}$ is 0.    
\end{lemma}

\begin{proof}
Suppose that the $p$-adic expansion of $k$ is $k=\sum_{i=0}^{f-1}k_ip^i$ where $0\le k_i\le p-1$. Then the $p$-adic expansion of $q-1-k$ is $q-1-k=\sum_{i=0}^{f-1}(p-1-k_i)p^i$. Thus, there is no carry in the base-$p$ addition of $k+(q-1-k)=q-1$, we see that $S_p(k)+S_p(q-1-k)=S_p(q-1)$, and so $\nu_p\binom{q-1}{k}=0$ by Theorem \ref{Kummer's theorem}.
\end{proof}

\begin{lemma}\label{valuation of c(q-1,k)}
For $1\le k\le q-1$, the $p$-adic valuation of $c(q-1,k)$ is 0.
\end{lemma}

\begin{proof}
By the proof of Lemma 3.1 in \cite{ward1990weight}, we know that for $1\le r,k\le q-1$, 
\begin{align*}
c(r,k)\equiv c(r,k;1)\pmod{pq}.   
\end{align*}   
When $r=q-1$, we consider  
\begin{align*}
c(q-1,k;1)=\binom{(q-1)q}{k}+\binom{(q-1)q}{(q-1)+k}+\cdots+\binom{(q-1)q}{a(q-1)+k}+\cdots    
\end{align*}
where the summation is restricted to $0\le a\le q-1$. We consider the summands in two cases:
\begin{enumerate}
    \item If $q\nmid a(q-1)+k$, then $\nu_p(a(q-1)+k)<f$ and so by Lemma \ref{the number of carries in certain z=x+y version 1},
\begin{align*}
\nu_p\binom{(q-1)q}{a(q-1)+k}\ge \nu_p\left((q-1)q\right)-\nu_p(a(q-1)+k)\ge 1.    
\end{align*} 
\item If $q|a(q-1)+k$, then $a=k$ since $0\le a\le q-1$. We have
\begin{align*}
\nu_p\binom{(q-1)q}{k(q-1)+k}&=\nu_p\binom{(q-1)q}{kq}\\
&=\nu_p\binom{q-1}{k} & \text{by Corollary \ref{shifted p-adic valuation}}\\
&=0 & \text{by Lemma \ref{valuation of binom}}  
\end{align*}
\end{enumerate}
It follows that $\nu_p(c(q-1,k;1))=0$, and so $\nu_p(c(q-1,k))=0$. 
\end{proof}

\begin{proposition}\label{valuation of c(r+p^j,p^j)}
Let $j,r$ be two integers such that $0\le j\le f-1$ and $0\le r\le q-1-p^j$. Then the $p$-adic valuation of $c(r+p^j,p^j)$ is equal to $\nu_p\binom{r+p^j}{p^j}$.
\end{proposition}

\begin{proof}
Again by the proof of Lemma 3.1 in \cite{ward1990weight}, $c(r,k)\equiv c(r,k;1)\pmod{pq}$, we have 
\begin{align*}
c(r+p^j,p^j)\equiv \binom{(r+p^j)q}{p^j}+\binom{(r+p^j)q}{(q-1)+p^j}+\cdots+\binom{(r+p^j)q}{a(q-1)+p^j}+\cdots\pmod{pq}    
\end{align*}  
where the summation is restricted to $0\le a\le q-1$. One term of the above sum is 
\begin{align*}
\binom{(r+p^j)q}{p^j(q-1)+p^j}=\binom{(r+p^j)q}{p^jq},
\end{align*}
which has the same $p$-adic valuation as $\binom{r+p^j}{p^j}$ by Corollary \ref{shifted p-adic valuation}. Thus, to prove that the $p$-adic valuation of $c(r+p^j,p^j)$ is equal to $\nu_p\binom{r+p^j}{p^j}$, it suffices to prove that for $a\ne p^j$, 
\begin{align*}
\nu_p\binom{(r+p^j)q}{a(q-1)+p^j}>\nu_p\binom{r+p^j}{p^j}.    
\end{align*}
Write $r$ in base $p$ so that $r=\sum_{i=0}^{f-1} r_ip^{i}$. Since $\nu_p\binom{r+p^j}{p^j}$ is equal to the number of carries when $p^j$ is added to $r$ in base $p$, by Theorem \ref{Kummer's theorem}, we see that 
\begin{align*}
s:=\nu_p\binom{r+p^j}{p^j}\le f-j-1.    
\end{align*}
Note that $s$ is equal to the number of consecutive $p-1$ starting from $r_j$, that is, $r_j=p-1,\ldots,r_{s+j-1}=p-1$, but $r_{s+j}\ne p-1$.\\

\noindent\textbf{Case 1}. When $\nu_p\binom{r+p^j}{p^j}=0$: 

Since 
\begin{align*}
\binom{(r+p^j)q}{a(q-1)+p^j}=\frac{(r+p^j)q}{a(q-1)+p^j}\binom{(r+p^j)q-1}{a(q-1)+p^j-1}  
\end{align*}
and for $a\ne p^j$, the numerator of $\frac{(r+p^j)q}{a(q-1)+p^j}$ is divisible by $q$ but the denominator of $\frac{(r+p^j)q}{a(q-1)+p^j}$ is not divisible by $q$, we see that 
\begin{align*}
\nu_p\binom{(r+p^j)q}{a(q-1)+p^j}\ge 1.    
\end{align*}
Hence, the $p$-adic valuation of $c(r+p^j,p^j)$ is equal to $\nu_p\binom{r+p^j}{p^j}$.\\

\noindent\textbf{Case 2}. If $\nu_p\binom{r+p^j}{p^j}>0$, then the $j$-th digit in the $p$-adic expansion of $r+p^j$ is 0.

\begin{enumerate}
    \item When $p^{j+1}\nmid a(q-1)+p^j$: 
    
    Since
\begin{align*}
\binom{(r+p^j)q}{a(q-1)+p^j}=\frac{(r+p^j)q}{a(q-1)+p^j}\binom{(r+p^j)q-1}{a(q-1)+p^j-1},
\end{align*}
we have 
\begin{align*}
\nu_p\binom{(r+p^j)q}{a(q-1)+p^j}\ge\nu_p\left(\frac{(r+p^j)q}{a(q-1)+p^j}\right)\ge f-j>s.    
\end{align*}
\item When $p^{j+1}\mid a(q-1)+p^j$ and $a\ne p^j$: 

In this case, $p^j-a\equiv a(q-1)+p^j\equiv0\pmod{p^{j+1}}$ since $p^{j+1}|q$, so we can write $a=a'p^{j+1}+p^j$ where $1\le a'<p^{f-j-1}$. Assume that $a'=bp^t$ where $p\nmid b$ and $t=\nu_p(a')$. Since $a=bp^{j+t+1}+p^j<q=p^f$, we have $f-j-1-t>0$.

We compute
\begin{align*}
\nu_p\binom{(r+p^j)q}{a(q-1)+p^j} &=\nu_p\binom{(r+p^j)q}{(a'p^{j+1}q+p^jq-a'p^{j+1}-p^j)+p^j} &\text{since $a=a'p^{j+1}+p^j$}\\
&= \nu_p\binom{(r+p^j)p^f}{bp^{j+1+t+f}+p^{j+f}-bp^{j+1+t}} &\text{since $a'=bp^t$}\\
&= \nu_p\binom{(r+p^j)p^{f-j-1-t}}{bp^f+(p^{f-1-t}-b)} &\text{by Corollary \ref{shifted p-adic valuation}}
\end{align*}

\begin{claim*}
We have
\begin{align*}
\nu_p\binom{(r+p^j)p^{f-j-1-t}}{bp^f+(p^{f-1-t}-b)}>s.    
\end{align*}
\end{claim*}

\textbf{Proof of the Claim.} Suppose that $b=b_{f-j-t-2}p^{f-j-t-2}+\cdots+b_0$ where $b_0\ne0$, or write it as $\left \langle b_{f-j-t-2},\ldots, b_0 \right \rangle_p$. Similarly, $r=r_{f-1}p^{f-1}+\cdots+r_0$ and by the assumption before, we have 
\begin{align*}
r=\langle \overbrace{r_{f-1},\ldots, r_{s+j}}^{f-s-j}, \overbrace{p-1,\ldots, p-1}^{s},\overbrace{r_{j-1},\ldots,r_0}^{j} \rangle_p.    
\end{align*}
Then 
\begin{align*}
r+p^j&= \langle \overbrace{r_{f-1},\ldots, r_{s+j}+1}^{f-s-j},\overbrace{ 0,\ldots, 0}^{s},\overbrace{r_{j-1},\ldots,r_0}^{j} \rangle_p, \\ 
(r+p^j)p^{f-j-1-t}&= \langle \overbrace{r_{f-1},\ldots, r_{s+j}+1}^{f-s-j},\overbrace{0,\ldots, 0}^{s},\overbrace{r_{j-1},\ldots,r_0}^{j},\overbrace{0,\ldots,0}^{f-j-1-t} \rangle_p.
\end{align*}
On the other hand, we have
\begin{align*}
&bp^f+(p^{f-1-t}-b)\\
=&\langle \overbrace{b_{f-j-t-2},\ldots,b_0}^{f-j-t-1},\overbrace{0,\ldots,0}^{t+1},\overbrace{p-1,\ldots,p-1}^{j}, \overbrace{p-1-b_{f-j-t-2},\ldots,p-1-b_1,p-b_0}^{f-j-t-1} \rangle_p
\end{align*}
Finally, applying Proposition \ref{the number of carries in certain z=x+y version 2} with $z=(r+p^j)p^{f-j-1-t}, x=bp^f+(p^{f-1-t}-b)$ and $y=z-x$ for $n=s+j+(f-j-1-t)=s+f-1-t$, we obtain
\begin{align*}
\nu_p\binom{(r+p^j)p^{f-j-1-t}}{bp^f+(p^{f-1-t}-b)}\ge n=s+f-1-t>s.
\end{align*}
\end{enumerate}
Now we can conclude that $\nu_p(c(r+p^j,p^j))=\nu_p\binom{r+p^j}{p^j}$.
\end{proof}

Finally, we are ready to prove Theorem \ref{First Theorem}. We restate Theorem \ref{First Theorem} below.

\begin{tthmbis}{First Theorem}
Let $C$ be a $[g_q(k,d),k,d]_q$ Griesmer code where $q=p^f$. If $q^e|d$, then $p^{e} | \mathrm{wt}(c)$ for all $c\in C$.
\end{tthmbis}

\begin{proof}[Proof of Theorem \ref{First Theorem}]
As before, we will apply Theorem \ref{divisibility criterion} with the basis $\{a_1,\ldots,a_k\}$ constructed in Theorem \ref{A basis for Griesmer codes} and the corresponding additive spanning set $B$, and then use the divisibility in (\ref{divisibility condition}) to prove Theorem \ref{First Theorem}. We proceed by induction on $k$. When $k\le e +1$, the conclusion is true by Proposition \ref{constant weight}. Assume that Theorem \ref{First Theorem} holds for Griesmer codes of dimension less than $k$, where $k\ge2$. Next we prove that Theorem \ref{First Theorem} holds for Griesmer codes of dimension $k$, or equivalently, the divisibility in (\ref{divisibility condition}) holds for all $(r_1,\ldots,r_k)$. Also, we denote $C_1=\mathrm{span}_{\mathbb{F}_q}\{a_2,\ldots,a_k\}$. The following proof will be divided into three steps.

\vspace{0.3cm}

\noindent\textbf{The First Step.} First, we will prove that when $r_1=q-1$, the divisibility in (\ref{divisibility condition}) holds. 

When $e=0$, Theorem \ref{First Theorem} is clearly true so the divisibility in (\ref{divisibility condition}) holds. Assume that $e\ge 1$ so $\mathrm{wt}(a_1+\alpha a_2)=d$ for any $\alpha\in\mathbb{F}_q$ by the argument before. When $r_1=q-1$, we observe
\begin{align*}
\sigma\left(T(A_1+\alpha' A_2)^{\circ(q-1)}\circ A_2^{\circ(r_2)}\circ\cdots\circ A_k^{\circ(r_k)}\right)=\sigma\left(\mathrm{proj}(A_2)^{\circ(r_2)}\circ\cdots\circ \mathrm{proj}(A_k)^{\circ(r_k)}\right),
\end{align*}
where $\mathrm{proj}(A_i)=T(\mathrm{proj}(a_i,a_1+\alpha a_2))$ with $i=2,\ldots,k$ and $\mathrm{proj}(a_i,a_1+\alpha a_2)$ is the projected codeword in $\mathrm{Proj}(C_1,a_1+\alpha a_2)$. Since $\mathrm{Res}(C,a_1+\alpha a_2)$ is an $\left[n-d,k-1,\frac{d}{q} \right]_q$ Griesmer code with $q^{e-1}|\frac{d}{q}$ by Theorem \ref{residual code and shortened code}, so it is $p^{e-1}$-divisible by induction. As a result, $\mathrm{Res}(C_1,a_1+\alpha a_2)$ is $p^{e-1}$-divisible. Note that the subcode $C_1$ is a $[g_q(k-1,d),k-1,d]_q$ Griesmer subcode, so it is $p^e$-divisible also by induction. Therefore, $\mathrm{Proj}(C_1,a_1+\alpha a_2)$ is $p^{e-1}$-divisible and hence
\begin{align*}
p^{e-1-\frac{1}{p-1}\sum_{i=2}^k S_p(r_i)+f} \;\big|\;\sigma\left(\mathrm{proj}(A_2)^{\circ(r_2)}\circ\cdots\circ \mathrm{proj}(A_k)^{\circ(r_k)}\right),   
\end{align*}
that is
\begin{align}\label{divisibility of r_1=q-1}
p^{e+f-1-\frac{1}{p-1}(S_p(q-1)+\sum_{i=2}^k S_p(r_i))+f} \;\big|\;\sigma\left(T(A_1+\alpha' A_2)^{\circ(q-1)}\circ A_2^{\circ(r_2)}\circ\cdots\circ A_k^{\circ(r_k)}\right),
\end{align}
which implies
\begin{align*}
p^{e-\frac{1}{p-1}(S_p(q-1)+\sum_{i=2}^k S_p(r_i))+f} \;\big|\;\sigma\left(T(A_1+\alpha' A_2)^{\circ(q-1)}\circ A_2^{\circ(r_2)}\circ\cdots\circ A_k^{\circ(r_k)}\right).    
\end{align*}
Hence the case $r_1=q-1$ is done.

\vspace{0.3cm}

\noindent\textbf{The Second Step.} Next we will prove that when $r_1+r_2\ge q-1$, the divisibility in (\ref{divisibility condition}) also holds. 

Suppose that $r_1+r_2\ge q-1$. Applying Proposition \ref{Expansion} with $X=T(A_1+ \alpha' A_2)$ and $Y=\beta' A_2$, we have
\begin{align*}
&\sigma\left(T(A_1+\alpha' A_2+\beta' A_2)^{\circ(q-1)}\circ A_2^{\circ(r_1+r_2-(q-1))}\circ A_3^{\circ(r_3)}\circ\cdots\circ A_k^{\circ(r_k)}\right) \\
=\;&\sigma\left(T(A_1+\alpha' A_2)^{\circ(q-1)}\circ A_2^{\circ(r_1+r_2-(q-1))}\circ A_3^{\circ(r_3)}\circ\cdots\circ A_k^{\circ(r_k)}\right)\\
&+(\beta')^{q-1}\sigma\left(A_2^{\circ(r_1+r_2-(q-1))}\circ A_3^{\circ(r_3)}\circ\cdots\circ A_k^{\circ(r_k)}\right)\\
&+\sum_{i=1}^{2q-r_1-r_2-2}(\beta')^i c(q-1,i)\sigma\left(T(A_1+\alpha' A_2)^{\circ(q-1-i)}\circ A_2^{\circ(r_1+r_2-(q-1)+i)}\circ A_3^{\circ(r_3)}\circ\cdots\circ A_k^{\circ(r_k)}\right)\\
&+\sum_{i=2q-r_1-r_2-1}^{q-2}(\beta')^{i} c(q-1,i)\sigma\left(T(A_1+\alpha' A_2)^{\circ(q-1-i)}\circ A_2^{\circ(r_1+r_2-2(q-1)+i)}\circ A_3^{\circ(r_3)}\circ\cdots\circ A_k^{\circ(r_k)}\right)\\
&+(\beta')^{q-1} c(q-1,q-1)\sigma\left(T(A_1+\alpha' A_2)^{\circ(q-1)}\circ A_2^{\circ(r_1+r_2-(q-1))}\circ A_3^{\circ(r_3)}\circ\cdots\circ A_k^{\circ(r_k)}\right).
\end{align*}
Since the divisibility in (\ref{divisibility condition}) holds for $r_1=0$ and $(\ref{divisibility of r_1=q-1})$ holds, we have
\begin{align*}
p^{s_1}\;&\big|\;\sigma\left(T(A_1+\alpha' A_2+\beta' A_2)^{\circ(q-1)}\circ A_2^{\circ(r_1+r_2-(q-1))}\circ A_3^{\circ(r_3)}\circ\cdots\circ A_k^{\circ(r_k)}\right)&\text{by (\ref{divisibility of r_1=q-1})}\\
p^{s_1} \;&\big|\;\sigma\left(T(A_1+\alpha' A_2)^{\circ(q-1)}\circ A_2^{\circ(r_1+r_2-(q-1))}\circ A_3^{\circ(r_3)}\circ\cdots\circ A_k^{\circ(r_k)}\right) &\text{by (\ref{divisibility of r_1=q-1})}\\
p^{s_2}\;&\big|\;\sigma\left(A_2^{\circ(r_1+r_2-(q-1))}\circ A_3^{\circ(r_3)}\circ\cdots\circ A_k^{\circ(r_k)}\right)&\text{by (\ref{divisibility condition})}
\end{align*}
where
\begin{align*}
s_1&=e+f-1-\frac{1}{p-1}\left(S_p(q-1)+S_p(r_1+r_2-(q-1))+\sum_{i=3}^k S_p(r_i)\right)+f\\
&=e-1-\frac{1}{p-1}\left( S_p(r_1+r_2-(q-1))+\sum_{i=3}^k S_p(r_i)\right)+f\\
s_2&=e-\frac{1}{p-1}\left(S_p(r_1+r_2-(q-1))+\sum_{i=3}^k S_p(r_i)\right)+f=s_1+1
\end{align*}
and hence for any $\beta\in\mathbb{F}_q^{\times}$, $p^{s_1}$ divides $S(\beta'):=\sum_{i=1}^{q-1}(\beta')^i x_i$, where
\begin{align*}
x_i&:=c(q-1,i)\sigma\left(T(A_1+\alpha' A_2)^{\circ(q-1-i)}\circ A_2^{\circ(r_1+r_2-(q-1)+i)}\circ A_3^{\circ(r_3)}\circ\cdots\circ A_k^{\circ(r_k)}\right),\\
x_j&:=c(q-1,j)\sigma\left(T(A_1+\alpha' A_2)^{\circ(q-1-j)}\circ A_2^{\circ(r_1+r_2-2(q-1)+j)}\circ A_3^{\circ(r_3)}\circ\cdots\circ A_k^{\circ(r_k)}\right),\\
x_{q-1}&:=c(q-1,q-1)\sigma\left(T(A_1+\alpha' A_2)^{\circ(q-1)}\circ A_2^{\circ(r_1+r_2-(q-1))}\circ A_3^{\circ(r_3)}\circ\cdots\circ A_k^{\circ(r_k)}\right),
\end{align*}
with $i=1,\ldots, 2q-r_1-r_2-2,j=2q-r_1-r_2-1,\ldots,q-2$. 

Therefore, for every $\beta\in\mathbb{F}_q^{\times}$, there exists a $c(\beta')\in\mathbb{Z}_p[\xi_{q-1}]$ such that $S(\beta')=c(\beta')\cdot p^{s_1}$. Consider the following linear system of equations in $x_i$'s: 
\begin{align*}
\begin{pmatrix}
 1 & 1 & \cdots & 1\\
 \xi_{q-1} & \xi_{q-1}^2 & \cdots & \xi_{q-1}^{q-1}\\
 \vdots & \vdots &  & \vdots\\
\xi_{q-1}^{q-2}  & \left(\xi_{q-1}^{q-2}\right)^2 & \cdots & \left(\xi_{q-1}^{q-2}\right)^{q-1}
\end{pmatrix}\begin{pmatrix}
x_1 \\
x_2 \\
\vdots \\
x_{q-1}
\end{pmatrix}=\begin{pmatrix}
S(1) \\
S(\xi_{q-1}) \\
\vdots \\
S\left(\xi_{q-1}^{q-2}\right)
\end{pmatrix}    
\end{align*}
Note that $\varphi$ maps units to $\mathbb{F}_q^{\times}$, $(p)$ to 0 and the above Vandermonde determinant is a unit in $\mathbb{Z}_p[\xi_{q-1}]$, as its image of $\varphi$ in $\mathbb{F}_q$ is nonzero. Thus, $p^{s_1}|x_i$ for $i=1,\ldots,q-1$. Since $p\nmid c(q-1,i)$ for $i=1,\ldots,2q-r_1-r_2-2$ by Lemma \ref{valuation of c(q-1,k)}, we obtain
\begin{align*}
p^{s_1} \;\big|\;\sigma\left(T(A_1+\alpha' A_2)^{\circ(q-1-i)}\circ A_2^{\circ(r_1+r_2-(q-1)+i)}\circ A_3^{\circ(r_3)}\circ\cdots\circ A_k^{\circ(r_k)}\right).
\end{align*}
In particular, for $i=q-1-r_1$, we get
\begin{align*}
p^{s_1} \;\big|\; \sigma\left(T(A_1+\alpha' A_2)^{\circ(r_1)}\circ A_2^{\circ(r_2)}\circ A_3^{\circ(r_3)}\circ\cdots\circ A_k^{\circ(r_k)}\right),
\end{align*}
which implies
\begin{align*}
p^{e-\frac{1}{p-1}\sum_{i=1}^k S_p(r_i)+f} \;\big|\; \sigma\left(T(A_1+\alpha' A_2)^{\circ(r_1)}\circ A_2^{\circ(r_2)}\circ A_3^{\circ(r_3)}\circ\cdots\circ A_k^{\circ(r_k)}\right)
\end{align*}
because the following equivalent inequalities hold (the explanations are given below):
\begin{align*}
&e-\frac{1}{p-1}\sum_{i=1}^k S_p(r_i)+f\le s_1=e-1-\frac{1}{p-1}\left(S_p(r_1+r_2-(q-1))+\sum_{i=3}^k S_p(r_i)\right)+f \\   
&\Longleftrightarrow\quad S_p(q-1)+S_p(r_1+r_2-(q-1))-(f-1)(p-1)\le S_p(r_1)+S_p(r_2) \\
&\Longleftrightarrow\quad 2S_p(q-1)-S_p(2(q-1)-r_1-r_2)-(f-1)(p-1)\le S_p(r_1)+S_p(r_2)\\
&\Longleftrightarrow\quad S_p(q-1-r_1)+S_p(q-1-r_2)-S_p(2(q-1)-r_1-r_2)\le (f-1)(p-1).
\end{align*}
The first equivalence follows from the direct simplification and the fact that $S_p(q-1)=f(p-1)$. The second equivalence follows from Lemma \ref{valuation of binom} with $(r_1+r_2-(q-1))+(2(q-1)-r_1-r_2)=q-1$ where $0\le r_1+r_2-(q-1)\le q-1$. Similarly, the third equivalence follows from Lemma \ref{valuation of binom} with $r_1+(q-1-r_1)=q-1$ and $r_2+(q-1-r_2)=q-1$ where $0\le r_1,r_2 \le q-1$. The last inequality holds because
\begin{align*}
\frac{S_p(q-1-r_1)+S_p(q-1-r_2)-S_p(2(q-1)-r_1-r_2)}{p-1}
\end{align*}
represents the number of carries in the base-$p$ addition of $(q-1-r_1)+(q-1-r_2)=2(q-1)-r_1-r_2$, which is less than $f$ since $2(q-1)-r_1-r_2\le q-1$.

This proves that the divisibility in (\ref{divisibility condition}) holds for $r_1+r_2\ge q-1$.

\vspace{0.3cm}

\noindent\textbf{The Third Step.} Now we use the backwards induction on $r_1$. Assume that for $r_1\ge r+1$, the divisibility in (\ref{divisibility condition}) holds. We will prove that for $r_1=r$, the divisibility in (\ref{divisibility condition}) holds. 

Let $\nu_p(r_2)=j$, i.e., $p^j||r_2$. Since when $r_2=0$, the divisibility in (\ref{divisibility condition}) clearly holds by induction hypothesis. Thus, it remains to consider $r_2\ge p^j$. If $r+p^j\ge q-1$, then $r+r_2\ge r+p^j\ge q-1$, and so the divisibility in (\ref{divisibility condition}) holds for $(r,r_2,\ldots,r_k)$. If $r+p^j<q-1$, since $r+p^j\ge r+1$, by induction the divisibility in (\ref{divisibility condition}) holds for $(r+p^j,r_2-p^j,r_3,\ldots,r_k)$, we know that 
\begin{align*}
p^{s_3} \;\big|\; \sigma\left(T(A_1+\alpha' A_2+\beta' A_2)^{\circ(r+p^j)}\circ A_2^{\circ(r_2-p^j)}\circ A_3^{\circ(r_3)}\circ\cdots\circ A_k^{\circ(r_k)}\right),    
\end{align*}
where $s_3=e-\frac{1}{p-1}(S_p(r+p^j)+S_p(r_2-p^j)+\sum_{i=3}^k S_p(r_i))+f$. By Proposition \ref{Expansion} and a similar argument of solving a system of linear equations, we see that
\begin{align*}
p^{s_3} \;\big|\; c(r+p^j,p^j)\sigma\left(T(A_1+\alpha' A_2)^{\circ(r)}\circ A_2^{\circ(r_2)}\circ A_3^{\circ(r_3)}\circ\cdots\circ A_k^{\circ(r_k)}\right).   
\end{align*}
By Theorem \ref{Kummer's theorem} and Proposition \ref{valuation of c(r+p^j,p^j)}, we have 
\begin{align*}
\nu_p(c(r+p^j,p^j))=\nu_p\binom{r+p^j}{p^j}=\frac{S_p(p^j)+S_p(r)-S_p(r+p^j)}{p-1}   
\end{align*}  
Therefore, we obtain
\begin{align*}
p^{e-\frac{1}{p-1}(S_p(r)+\sum_{i=2}^k S_p(r_i))+f} \;\big|\; \sigma\left(T(A_1+\alpha' A_2)^{\circ(r)}\circ A_2^{\circ(r_2)}\circ A_3^{\circ(r_3)}\circ\cdots\circ A_k^{\circ(r_k)}\right)    
\end{align*}
since 
\begin{align*}
S_p(r+p^j)+S_p(r_2-p^j)+S_p(p^j)+S_p(r)-S_p(r+p^j)= S_p(r)+S_p(r_2)    
\end{align*}
where we have used $\nu_p(r_2)=j$. Thus, the divisibility in (\ref{divisibility condition}) holds for $(r,r_2,r_3,\ldots,r_k)$.  

The proof of Theorem \ref{First Theorem} is now complete.
\end{proof}

\section{The Geometric Proof of Theorem \ref{Second Theorem}}\label{Geometric Approach}

We also need a lemma which will play a similar role as the following inheritance property of the divisibility. 

\begin{proposition}[Lemma 3.4 in \cite{kurz2021divisible}]\label{Divisibility inherits}
Let $C$ be an $[n,k]_q$ code with effective length $n$ and $\mathcal{M}_C$ be the corresponding multiset. If $C$ is $\Delta$-divisible for some $\Delta$ divisible by $q$, then the multiset $\mathcal{M}_C|_H$ is $\Delta/q$-divisible for any hyperplane $H\in\mathcal{H}$ in $\mathrm{PG}(k-1,q)$.
\end{proposition}

If the multiset $\mathcal{M}$ is $\Delta$-divisible with $q^r|\Delta$. Then by induction, we can prove that the multiset $M|_U$ is $\Delta/q^j$-divisible for any subspace $U$ of codimension $j\le r$. 

Below, we will use an embedding from $\mathrm{PG}(k-2, q)$ to $\mathrm{PG}(k-1, q)$ by adding a zero in the first coordinate of each projective point.

\begin{lemma}\label{Grisemer code restricted to a constant weight hyperplane}
Let $C$ be a $[g_q(k,d),k,d]_q$ Griesmer code and $G$ be a generator matrix of $C$. Suppose that $h=aG\in C$ is a minimum weight codeword where $a^{\top}\in\mathbb{F}_q^k$, and $H_a:=a^{\perp}$ is the hyperplane consisting of vectors that are orthogonal to $a$. Then $\mathcal{M}_C|_{H_a}$ is projectively equivalent to a Griesmer multiset corresponding to the $[g_q(k-1,\left \lceil d/q \right \rceil),k-1,\left \lceil d/q \right \rceil]_q$ Griesmer code $\mathrm{Res}(C,h)$, where the latter is considered to be embedded in $\mathrm{PG}(k-1,q)$.   
\end{lemma}

\begin{proof}
Since $\mathrm{GL}_k(q)$ acts transitively on $\mathbb{F}_q^k\backslash\{0\}$, there exists an $R\in\mathrm{GL}_k(q)$ such that 
\begin{align*}
a=\overbrace{(1,0,\ldots,0)}^{k}R = e_1^{\top}R.        
\end{align*}
Denote $G=(g_1,g_2,\ldots,g_n)$ and $G' = RG = (g_1',g_2'.\ldots,g_n')$, where $g_i' = Rg_i$. Note that $G'$ is another generator matrix of $C$ and the first row of $G'$ is the minimum weight codeword $h$, thus for $i\notin \mathrm{supp}(h)$, the first entry of $g_i'$ is zero. Let $G''$ be the matrix obtained by deleting the first row and the columns indexed by $\mathrm{supp}(h)$ from $G'$. Then $G''$ is a generator matrix of $\mathrm{Res}(C,h)$, which is a $[g_q(k-1,\left \lceil d/q \right \rceil),k-1,\left \lceil d/q \right \rceil]_q$ Griesmer code. 

Let $\mathcal{M}_{G},\mathcal{M}_{G'}$ and $\mathcal{M}_{G''}$ be the corresponding multisets of $G,G'$ and $G''$, respectively. We claim that $\mathcal{M}_{G'}|_{H_{e_1^{\top}}}$ is the image of $\mathcal{M}_{G''}$ via the embedding from $\mathrm{PG}(k-2,q)$ to $\mathrm{PG}(k-1,q)$, since
\begin{align*}
\mathcal{M}_{G'}|_{H_{e_1^{\top}}}=\left\{\left\langle g_i' \right\rangle :  e_1^{\top} g_i'=0 \right\}=\left\{\left\langle g_i' \right\rangle : g_i'(1) = 0 \right\}.
\end{align*}

It remains to show that $\mathcal{M}_G|_{H_a}$ is projectively equivalent to $\mathcal{M}_{G'}|_{H_{e_1^{\top}}}$. Note that
\begin{align*}
    \mathcal{M}_G|_{H_a}=\{\left\langle g_i \right\rangle : a g_i=0 \} =\{\left\langle g_i \right\rangle : e_1^{\top}g_i'=0 \},  
\end{align*}
so $\mathcal{M}_{G'}|_{H_{e_1^{\top}}} = R\mathcal{M}_G|_{H_a}$, and thus $\mathcal{M}_G|_{H_a}$ is  projectively equivalent to $\mathcal{M}_{G'}|_{H_{e_1^{\top}}}$.
\end{proof}

We are now in a position to give the proof of Theorem \ref{Second Theorem} using the geometric approach. Below we restate Theorem \ref{Second Theorem} for the convenience of the reader.

\begin{tthmbis}{Second Theorem}
Let $C$ be a $[g_q(k,d),k,d]_q$ Griesmer code where $q=p^f$. If $p^e|d$, then $\left\lceil p^{e-(f-1)(q-2)} \right\rceil$ is a divisor of $C$.    
\end{tthmbis}

\begin{proof}[Proof of Theorem \ref{Second Theorem}]
First, when $e\le(f-1)(q-2)$, the conclusion immediately follows since $\left\lceil p^{e-(f-1)(q-2)} \right\rceil=1$. 

It remains to consider the case when $e>(f-1)(q-2)$, we consider the dimension $k$ in two cases as follows.

\vspace{0.3cm}

\noindent\textbf{Case 1.} $k\le q-1$: 

\textbf{Subcase 1-1.} If $e\ge f(q-2)$ we have $q^{k-1}|q^{q-2}|p^e|d$. Then, by Proposition \ref{constant weight}, $C$ is a constant weight code so $p^e|\mathrm{wt}(c)$ for all $c\in C$. 

\textbf{Subcase 1-2.} If $(f-1)(q-2)<e<f(q-2)$, then by Theorem \ref{First Theorem}$'$, we have $p^{\lfloor e/f\rfloor}|\mathrm{wt}(c)$ for all $c\in C$. Since
\begin{align*}
\frac{e}{f}\ge e-(f-1)(q-2),    
\end{align*}
we see that $p^{e-(f-1)(q-2)} |\mathrm{wt}(c)$ for all $c\in C$.

\vspace{0.3cm}

\noindent\textbf{Case 2.} $k\ge q$: 

\textbf{Subcase 2-1.} If $(f-1)(q-2)<e<f(q-2)$, by the same argument as above, we can prove that $p^{e-(f-1)(q-2)} |\mathrm{wt}(c)$ for all $c\in C$.

\textbf{Subcase 2-2.} Let $e\ge f(q-2)$. We assume that the conclusion is true for Griesmer codes of dimension less than $k$. By Theorem \ref{A basis for Griesmer codes}, we can choose a generator matrix $G$ of $C$ such that any $k-1$ rows of $G$ span a $[g_q(k-1,d),k-1,d]_q$ Griesmer subcode of $C$. According to Corollary \ref{congruence corollary}, the goal is to prove that for the multiset $\mathcal{M}$ corresponding to the columns of $G$, the following congruence
\begin{align}\label{congruence equation 1}
\mathcal{M}(H)\equiv n\pmod{p^{e-(f-1)(q-2)}}
\end{align}
holds for every hyperplane $H\in\mathcal{H}$. 

Since Theorem \ref{Second Theorem} holds for Griesmer codes of dimension less than $k$ by the induction hypothesis, we can apply this to these $k$ Griesmer subcodes $C_i$, where $C_i$ is spanned by all rows of $G$ except the $i$-th row, $i=1,\ldots,k$. For any $h=(h_1,\ldots,h_k)$ with some $h_i=0$, since $h G\in C_i$, we have $p^{e-(f-1)(q-2)}|\mathrm{wt}(h G)$. Define
\begin{align*}
H_{b}=\{(x_1,x_2,\ldots,x_k)\in\mathbb{F}_q^k:b_1x_1+b_2x_2+\cdots+b_kx_k=0\},    
\end{align*}
where $b=(b_1,b_2,\ldots,b_k)$. Then by Proposition \ref{weight},
\begin{align*}
n-\mathcal{M}(H_{h})=\mathrm{wt}(h G)\equiv0\pmod{p^{e-(f-1)(q-2)}},
\end{align*}
thus congruence (\ref{congruence equation 1}) holds for the hyperplane $H_{h}$.

Next, we will prove that for any hyperplane $H_{u}$ where $u=(u_1,\ldots,u_k)$ with $u_i\ne0$ for all $i$, congruence $(\ref{congruence equation 1})$ still holds. Suppose that $\mathbb{F}_q=\{\gamma_1,\gamma_2,\ldots,\gamma_{q}=0\}$. Since $k\ge q$, consider the following system of two linear equations:
\begin{numcases}{}
u_1 x_1+u_2 x_2+\cdots+u_k x_k&=\;\;0 \label{equation 1}\\
\gamma_1 u_1 x_1+\gamma_2 u_2 x_2+\cdots+\gamma_{q-1} u_{q-1} x_{q-1}&=\;\;0 \label{equation 2}
\end{numcases}
which defines a $(k-2)$-dimensional subspace $H''$. Note that there are totally ${k-(k-2)\brack (k-1)-(k-2)}_q=q+1$ hyperplanes containing $H''$ and two of them are the hyperplanes $H_{u}$ and $H_q$ defined by (\ref{equation 1}) and (\ref{equation 2}) respectively. The remaining $q-1$ of them are exactly the hyperplane $H_i$ defined by the equation $(\ref{equation 2})-\gamma_i\times(\ref{equation 1})$, where $i=1,2,\ldots,q-1$; that is: 
\begin{align*}
H_i:=\left\{(x_1,x_2,\ldots,x_k)\in\mathbb{F}_q^k: \sum_{i=1}^{q-1}\gamma_i u_ix_i-\gamma_i \sum_{j=1}^k u_j x_j=0 \right\}, \quad i=1,2,\ldots,q-1.
\end{align*}
Moreover, by induction we see that the congruence (\ref{congruence equation 1}) holds for hyperplanes $H_1,H_2,\ldots,H_q$ because the coefficient of $x_i$ in $(\ref{equation 2})-\gamma_i\times(\ref{equation 1})$ is exactly 0 as $k\ge q$. Therefore, we have
\begin{align}\label{divisibility of certain hyperplanes}
\mathcal{M}(H_i)\equiv n\pmod{p^{e-(f-1)(q-2)}},\quad i=1,2,\ldots,q.
\end{align}
Note that $p^{e}|d$ with $e\ge f(q-2)$, so $q^{q-2}|d$. Then by Theorem \ref{A basis for Griesmer codes}, the first $(q-1)$ rows $\{a_1,a_2,\ldots,a_{q-1}\}$ span a Griesmer subcode with constant weight $d$, and hence 
\begin{align*}
(\gamma_1 u_1,\gamma_2 u_2,\ldots,\gamma_{q-1} u_{q-1},0,\ldots,0)G\in\mathrm{span}_{\mathbb{F}_q}\{a_1,a_2,\ldots,a_{q-1}\}    
\end{align*}
has minimum weight $d$. By Lemma \ref{Grisemer code restricted to a constant weight hyperplane}, we know that the multiset $\mathcal{M}|_{H_q}$ corresponds to a $[g_q(k-1,d/q),k-1,d/q]_q$ Griesmer code, which is $\left\lceil p^{e-(f-1)(q-2)}/q\right\rceil$-divisible by induction. Therefore, we have
\begin{align*}
\mathcal{M}|_{H_q}(H_1)\equiv\#\mathcal{M}|_{H_q}= \mathcal{M}(H_q)\equiv n\pmod{\left\lceil p^{e-(f-1)(q-2)}/q\right\rceil}.    
\end{align*}
It follows that
\begin{align}\label{divisibility inherits for overlap subspace}
\mathcal{M}(H'')=\mathcal{M}(H_q\cap H_1)=\mathcal{M}|_{H_q}(H_1)\equiv n\pmod{\left\lceil p^{e-(f-1)(q-2)}/q\right\rceil}.
\end{align}
Note that
\begin{align}\label{overlap of q+1 hyperplanes}
n=\#\mathcal{M}=\mathcal{M}(H_{u})+\sum_{i=1}^{q}\mathcal{M}(H_i)-q\mathcal{M}(H''). 
\end{align}
Combining congruences (\ref{divisibility of certain hyperplanes}), (\ref{divisibility inherits for overlap subspace}) and equality (\ref{overlap of q+1 hyperplanes}), we obtain 
\begin{align*}
\mathcal{M}(H_{u})\equiv n\pmod{p^{e-(f-1)(q-2)}}    
\end{align*}
which shows that any hyperplane will satisfy congruence (\ref{congruence equation 1}), and hence the conclusion is true for Griesmer codes of dimension $k$. This completes the proof of Theorem \ref{Second Theorem}.
\end{proof}

By the geometric argument in the proof of Theorem \ref{Second Theorem}, we can reduce Conjecture \ref{Conjecture 1} to the case when $k\le q-1$. In other words, we have the following result.

\begin{theorem}\label{reduction of Ward's conjecture}
If Conjecture \ref{Conjecture 1} holds for any $[g_q(k,d),k,d]_q$ Griesmer code with $k\le q-1$, then it will be true for all Griesmer codes over $\mathbb{F}_q$.
\end{theorem}

\section{Summary and Discussions}

\subsection{Summary}

Let $C$ be a $[g_q(k,d),k,d]_q$ Griesmer code where $q=p^f$. Suppose that $q^e|d$. In this paper, we have found a basis $\{a_1,a_2,\ldots,a_k\}$ of $C$ such that $\{a_1,a_2,\ldots,a_{\min\{e+1,k\}}\}$ span a constant weight code. Moreover, any $k-1$ of them span a $[g_q(k-1,d),k-1,d]_q$ Griesmer subcode. In 1998, Ward applied Theorem \ref{divisibility criterion} to prove Theorem \ref{Divisibility over F_p}. However, Ward only considered Griesmer codes over the prime field $\mathbb{F}_p$ (rather than over $\mathbb{F}_q$, where $q$ is a proper prime power). In the case where $q$ is a proper prime power, the $p$-adic valuation of the coefficients $c(r,k)$ becomes more difficult to analyze. We are able to overcome these difficulties and prove Theorem \ref{First Theorem}. Secondly, using this basis together with Theorem \ref{First Theorem}, we are able to prove Theorem \ref{Second Theorem} from the geometric point of view. 

%The idea behind the formulation of Theorem \ref{Second Theorem} is that the divisibility of Griesmer codes of dimension $k\ge q$ can be reduced to the divisibility of Griesmer codes of dimension less than $q$. On the other hand, by Theorem \ref{First %Theorem}, we know that $q^e|d$ implies $p^e$ is a divisor of $C$. The divisibility transition from $p^e|d$ to $p^{e-\delta} | \mathrm{wt}(c)$ for all nonzero $c \in C$ (where $\delta \geq 0$ quantifies the potential divisibility loss) can be bounded as follows: when %$q^{q-3}p^{f-1} \mid d$, the weight divisibility satisfies
%\begin{align*}
%p^{q-3} | \mathrm{wt}(c), \quad \forall c \in C ,
%\end{align*}
%with the maximal possible loss factor being 
%\begin{align*}
%\frac{q^{q-3}p^{f-1}}{p^{q-3}} = p^{(f-1)(q-2)}.
%\end{align*}

% Hence, from $p^e|d$ to $p^{?}|\mathrm{wt}(c)$ for all $c\in C$, the divisibility of the Griesmer code $C$ can lose at most \begin{align*}
% q^{q-3}p^{f-1}/p^{q-3}=p^{(f-1)(q-2)},     
% \end{align*}
% when the divisibility from $q^{q-3}p^{f-1}|d$ to $p^{q-3}|\mathrm{wt}(c)$ for all $c\in C$.

\subsection{Discussions}

Recall that Ward's conjecture \cite{ward2001divisible} can be restated as follows:

\begin{conjbis}{Conjecture 1}
Let $C$ be a $[g_q(k,d),k,d]_q$ Griesmer code where $q=p^f$. If $p^e|d$, then $\left\lceil p^{e-(f-1)} \right\rceil$ is a divisor of $C$.  
\end{conjbis}

By Theorem \ref{reduction of Ward's conjecture}, we reduce Ward's conjecture to the case when $k\le q-1$; so the problem is now simplified as follows. 

\begin{problem}
Prove or disprove Ward's conjecture for $k\le q-1$.    
\end{problem}

Previously, the best divisibility results were Theorem \ref{Divisibility over F_p} and Theorem \ref{evidences}. Currently, the state of the art is given by Theorem \ref{First Theorem}, Theorem \ref{Second Theorem} and Theorem 14 of \cite{ward2001divisible}. 

We believe that the truth is closer to Theorem \ref{First Theorem} and Theorem \ref{Second Theorem} than to Ward's conjecture; in other words, Ward's conjecture might be false. The evidence to support our guess is that in the proof of Theorem \ref{First Theorem}, the inequalities can hold as equalities for specific cases. 

% In the case when $q=4$, Ward's conjecture is true probably because it suffices to examine that the divisibility of all $[n,k,d]_q$ Griesmer codes in Ward's conjecture indeed holds for $k\le q-1=3$ (the dimension is small, also, we note that the classification of $[n,3,d]_q$ Griesmer codes is known if $q|d$ \cite{hill2007geometric}).

If we are trying to prove Ward's conjecture, the method in the proof of Theorem \ref{First Theorem} seems difficult to improve. 

On the other hand, in order to disprove Ward's conjecture, we need to find a counterexample. The authors suggest searching for a $[g_q(k,d),k,d]_q$ Griesmer code with parameters $k,d,q=p^f$ subject to 
\begin{align*}
q:&\quad\text{non-prime }q\ge 8,\\
k:&\quad 4\le k\le q-1,\\
d:&\quad f+1\le\nu_p(d)<\min\{f(q-2),f(k-1))\},    
\end{align*}
which may provide a counterexample to Ward's conjecture. It turns out that Griesmer codes satisfying the above conditions may not exist when the code parameters are relatively small. This fact presents a challenge to finding a counterexample. 

%In the end, we refer the interested readers to consult the following sources: \url{https://web.mat.upc.edu/simeon.michael.ball/codebounds.html} and \url{https://mars39.lomo.jp/opu/griesmer.htm}.

\section*{Acknowledgments}

The authors would like to thank Simeon Ball for valuable discussions and suggestions. We also thank Zhen Jia for many discussions.

\bibliographystyle{amsplain}

{\small
Department of Mathematics and National Center for Applied Mathematics Shenzhen, Southern University of Science and Technology, Shenzhen 518055, China 

\text{Email address:} {\href{mailto:12131225@mail.sustech.edu.cn}{\textcolor{black}{12131225@mail.sustech.edu.cn}}}

\vspace{0.2cm}
School of Cyber Science and Technology, Shandong
University, Qingdao 266000, China 

\text{Email address:} {\href{mailto:hexianghuang@foxmail.com}{\textcolor{black}{hexianghuang@foxmail.com}}}

\vspace{0.2cm}
Department of Mathematics and Shenzhen International Center of Mathematics, Southern University of \text{Science} and Technology, Shenzhen 518055, China

\text{Email address:} {\href{mailto:xiangq@sustech.edu.cn}{\textcolor{black}{xiangq@sustech.edu.cn}}}
}

\end{document}